\documentclass{aims}
\usepackage{amsmath}
  \usepackage{graphics} 
  \usepackage{epsfig} 
\usepackage{graphicx}  \usepackage{epstopdf}
 \usepackage[colorlinks=true]{hyperref}
\hypersetup{urlcolor=blue, citecolor=red}

\def\currentvolume{X}
 \def\currentissue{X}
  \def\currentyear{200X}
   \def\currentmonth{XX}
    \def\ppages{X--XX}
     \def\DOI{10.3934/xx.xx.xx.xx}

\newtheorem{theorem}{Theorem}[section]
\newtheorem{corollary}{Corollary}

\newtheorem{lemma}[theorem]{Lemma}
\newtheorem{proposition}{Proposition}

\theoremstyle{definition}
\newtheorem{definition}[theorem]{Definition}
\newtheorem{remark}{Remark}

\title[Two-species coagulation]
      {Kinetic theory and numerical simulations of two-species coagulation}

\author[C. Escudero, F. Maci\`{a}, R. Toral, and J. J. L. Vel\'azquez]{}

\subjclass{Primary: 35C20, 35Q20; Secondary: 82C22, 82C80.}
 \keywords{Smoluchowsky equations, self-similar asymptotics, coagulation, generating functions, numerical experiments.}

 \email{cel@icmat.es}
 \email{fabricio.macia@upm.es}
 \email{raul@ifisc.uib-csic.es}
 \email{velazquez@iam.uni-bonn.de}

\thanks{This work has been
partially supported by the MICINN (Spain) and FEDER (EU) through
Projects No. RYC-2011-09025, No. MTM2010-18128,  MTM2007-61755,
MTM2008-03754 and FIS2007-60327.}

\begin{document}
\maketitle

\centerline{\scshape Carlos Escudero }
\medskip
{\footnotesize
 \centerline{Departamento de Matem\'{a}ticas \& ICMAT (CSIC-UAM-UC3M-UCM)}
   \centerline{Universidad Aut\'{o}noma de Madrid, Ciudad Universitaria de Cantoblanco}
   \centerline{28049 Madrid, Spain}
} 

\medskip

\centerline{\scshape Fabricio Maci\`{a} }
\medskip
{\footnotesize
 \centerline{Universidad Polit\'{e}cnica de Madrid}
   \centerline{ETSI Navales, Avda. Arco de la Victoria s/n}
   \centerline{28040 Madrid, Spain}
} 

\medskip

\centerline{\scshape Ra\'ul Toral }
\medskip
{\footnotesize
 \centerline{IFISC (Instituto de F\'{\i}sica Interdisciplinar y Sistemas Complejos)}
   \centerline{CSIC-UIB, Campus UIB}
   \centerline{07122 Palma de Mallorca, Spain}
} 

\medskip

\centerline{\scshape Juan J. L. Vel\'{a}zquez }
\medskip
{\footnotesize
 \centerline{Hausdorff Center for Mathematics}
   \centerline{Rheinischen Friedrich-Wilhelms-Universit\"{a}t Bonn}
   \centerline{53115 Bonn, Germany}
} 

\bigskip

 \centerline{(Communicated by the associate editor name)}

\begin{abstract}
In this work we study the stochastic process of two-species
coagulation. This process consists in the aggregation dynamics
taking place in a ring. Particles and clusters of particles are set
in this ring and they can move either clockwise or counterclockwise.
They have a probability to aggregate forming larger clusters when
they collide with another particle or cluster. We study the
stochastic process both analytically and numerically. Analytically,
we derive a kinetic theory which approximately describes the process
dynamics. One of our strongest assumptions in this respect
is the so called well--stirred limit, that allows neglecting the
appearance of spatial coordinates in the theory, so this becomes
effectively reduced to a zeroth dimensional model.
We determine the long time behavior of such a model, making emphasis
in one special case in which it displays self-similar solutions.
In particular these calculations
answer the question of how the system gets ordered, with all
particles and clusters moving in the same direction, in the long
time. We compare our analytical results with direct numerical
simulations of the stochastic process and both corroborate its
predictions and check its limitations. In particular, we numerically
confirm the ordering dynamics predicted by the kinetic theory and
explore properties of the realizations of the stochastic process
which are not accessible to our theoretical approach.
\end{abstract}

\section{Introduction}\label{introduction}

The theoretical study of coagulation and their kinetic description
is of broad interest because of its vast applicability in diverse
topics such as aerosols \cite{seinfeld}, polymerization
\cite{ziff,sintes}, Ostwald ripening \cite{ls,meerson}, galaxies and
stars clustering \cite{silk}, and population biology \cite{niwa}
among many others. In this work we propose a generalization of the
stochastic process of coagulation. Herein we will consider the
coagulation process among two different species: the aggregation
will take place only when one element of one of the species
interacts with an element of the other species. In particular, we
will place particles and clusters of particles in a ring, where they
will move with constant speed, either clockwise or counterclockwise.
When two clusters (or two particles or one particle and one cluster)
meet they have the chance to aggregate and form a cluster containing
all particles involved in the collision. The direction of motion of
the newborn cluster is chosen following certain probabilistic rules.
We are interested in the properties of the realizations of such a
stochastic process, and in particular in their long time behavior.
Our main theoretical technique will be the use of kinetic equations,
an approach we have outlined in~\cite{emv}. This is of course just a
possible extension of the theory of coagulation. We have designed it
getting inspiration from self-organizing systems and in particular
from collective organism behavior. Let us note that this is a field
that has been studied using a broad range of different theoretical
techniques \cite{toralmarro,bertozzi,bodnar,orsogna,vicsek}. Another
field which has inspired ourselves is the study of the dynamics of
opinion formation and spreading~\cite{deffuant,toral,fortunato},
which is represented for instance by the classical voter
model~\cite{voter1,voter2,liggett}. As a final influence, we mention
that clustering has been previously studied in population dynamics
models~\cite{emilio} including swarming systems~\cite{clusterswarm},
and coagulation equations have been used in both
swarming~\cite{coaguswarm} and opinion formation
models~\cite{toral2}. Despite of its simplicity, the two-species
coagulation model could be related to some of these systems.

A particular system that has influenced the current developments
is the collective motion of locusts. The
experiment performed in \cite{buhl} revealed that locusts marching
on a (quasi one dimensional) ring presented a coherent collective
motion for high densities; low densities were characterized by a
random behavior of the individuals and intermediate densities
showed coherent displacements alternating with sudden changes of
direction. The models that have been introduced to describe this
experiment assume that the organisms behave like interacting
particles~\cite{buhl,yates,escudero}. Related interacting particle models
have been used to describe the collective behavior of many
different organisms and analyzing the mathematical properties of
such models has been a very active research
area~\cite{bertozzi,orsogna,vicsek}. The two-species coagulation model could
be thought of as a particular limit of some of these models or as a simplified
version of them which still retains some desirable features.

The goal of the current work is not to describe the detailed behavior of
any specific system. Instead, we will explore the mathematical properties
of a stochastic process which has been designed by borrowing inspiration from different
self-organizing systems. Therefore the focus will be on mathematical tractability.
We will determine under which conditions consensus is reached and what form it adopts.
In order to do this we introduce a kinetic theory that approximately describes the stochastic process.
One of our main assumptions is to consider that the system is well-stirred, so the explicit dependence on the
spatial coordinates can be dropped. This approximation is very common in other fields like in the
study of reaction-diffusion systems or Boltzmann equations. The stochastic process, which will be just qualitatively
described in this introduction, will be precisely introduced in section~\ref{numericalresults}.
The kinetic theory we will consider in order to theoretically describe this process, equation~\eqref{smoluchowsky} below,
will be heuristically introduced rather than derived as a proper limit of the stochastic process. Our present approach is
an {\it aposterioristic} one: we will compare the theoretical predictions of our kinetic equation with direct numerical simulations
of the stochastic process in order to check the validity of our theory.

In the following section we describe the kinetic theory that approximately describes the stochastic process under study in the
well-stirred limit. We build our progress here based on our previous developments in~\cite{emv}.
The basic analysis of the resulting kinetic equations is carried out in this section too, and
the results obtained are interpreted in the context of the underlying stochastic process. Part of the phenomenology of
the ordering dynamics is already unveiled at this point.
We devote the subsequent sections to a more detailed mathematical analysis of this kinetic theory. In particular,
we study the advent of self-similarity of the solutions to the kinetic equations. We interpret this self-similar
behavior as the formation of a giant cluster composed by all the particles in the system.

We postpone to the last section of the paper the direct numerical simulations
of the stochastic process. In this last section we both check the predictions of
our kinetic theory and explore the stochastic process beyond the kinetic level. Kinetic
approximations neglect many sources of fluctuations and thus numerical simulations are
required in order to describe many properties of the individual realizations of the stochastic
process which are not reflected at the kinetic level. It is also important to remark here that
the direct numerical simulations of the stochastic process do not assume the well-stirred limit.
So in these simulations the spatial distribution of particles and clusters is explicitly taken
into account. Let us mention again that the only reason why the well-stirred limit is considered
in the theoretical part of this work is in order to allow for mathematical tractability, so there is
no reason in imposing such a constraint in the numerical part. Furthermore, the numerical results
justify the theoretical assumption under appropriate hypothesis on the model parameters.

Our analysis will be limited to the one--dimensional spatial
situation with periodic boundary conditions (the dynamics is
taking place in a circumference). Our present approach will be
based on coagulation equations. This kinetic description is in general invalid
for one-dimensional systems because it is known that spatial correlations do propagate
in this dimensionality. So we assume a collision takes place when two clusters meet with a very small probability.
The probability should be so small that all the particles travel the whole system several times before one
collision happens on average. This way the system becomes well--stirred, and so we can neglect spatial correlations
and treat the system as if it were zero dimensional, what allows mathematical tractability. The rate at which the
collisions take place can be trivially absorbed by means of a rescaling of time, so we do not include it explicitly in
the formulation of the kinetic equations. It will however be explicitly taken into account when we perform direct
numerical simulations of the stochastic process in section \ref{numericalresults}.

We assume that the particles move in clusters of $\ell$ individuals
and $f^+(\ell,t)$ ($f^-(\ell,t)$) represents the number of clusters
of size $\ell$ moving clockwise (counterclockwise) at time $t$. We
note these functions should in principle only take integer values,
however the kinetic approximation implicitly averages over many
realizations, so they will be allowed to take any non-negative real
value. Clusters are modified when they collide with other clusters,
in such a way that the probability distributions obey the following
equations of motion:
\begin{eqnarray} \label{smoluchowsky}
\partial_t f^\mp(\ell,t) = \sum_{m,k,j=1}^\infty \left[ \Psi(k,j;m,\ell)
\delta_{k+j,m+\ell} f^\pm(k,t)f^\mp(j,t)- \right. \\ \left.
\Psi(m,\ell;k,j) \delta_{m+\ell,k+j}f^\pm(m,t) f^\mp (\ell,t)
\right], \nonumber
\end{eqnarray}
where $\Psi(k,j;m,\ell)$ is the collision kernel: It states the
probability with which a collision among clusters with $k$ and $j$
particles occurs and yields clusters with $m$ and $\ell$
particles.
One of our basic assumptions is its symmetry under
reflections $\Psi(k,j;m,\ell)=\Psi(j,k;\ell,m)$.
It is worth remarking that equation~\eqref{smoluchowsky} is not the exact description
of the coagulation process. It is in fact an approximation, but we will not derive it
from the stochatic process. Our approach will be phenomenological: we postulate the
validity of such a kinetic description and check it {\it a posteriori} by means of
numerical simulations. Let us however briefly comment on that it is reasonable postulating
this kinetic theory. The right hand side in the upper line of~\eqref{smoluchowsky} is the {\it gain}
term: it represents the possibility of formation of clusters of size $\ell$ from the collision of
a cluster of size $k$ and another one of size $j$; the sum takes into account all possible sizes $j$ and $k$.
The second line of this equation is the {\it loss} term: it takes into account the chance of disappearance of a cluster
of size $\ell$ by its coagulation with a cluster of any other size.

Our analytical
progress on this equation will be built by means of the
introduction of the generating functions
\begin{equation}
F^\pm (z,t)=\sum_{\ell=1}^\infty f^\pm(\ell,t) z^\ell.
\end{equation}
If we take the derivative these functions with respect to time we get
\begin{eqnarray}\label{substitution}
\partial_t F^\pm (z,t)=\sum_{\ell=1}^\infty \partial_t f^\pm(\ell,t) z^\ell = \\ \nonumber
\sum_{\ell,m,k,j=1}^\infty \left[ \Psi(k,j;m,\ell)
\delta_{k+j,m+\ell} f^\pm(k,t)f^\mp(j,t)- \right. \\ \nonumber \left.
\Psi(m,\ell;k,j) \delta_{m+\ell,k+j}f^\pm(m,t) f^\mp (\ell,t)
\right] z^\ell. \nonumber
\end{eqnarray}
As already mentioned, the focus of this paper is on mathematical tractability, and so we will only consider
kernels such that this expression becomes a closed system of differential equations for the generating functions.
Our two choices are described in the following.

\subsection{Random kernel}

In the present paper we will concentrate on two integrable
kernels. Our first choice is
\begin{eqnarray} \label{collisionkernel1}
\Psi(\ell,m;0,m+\ell) &=& \Psi(\ell,m;m+\ell,0)=1/2, \\
\Psi(k,j,\ell,m) &=& 0, \qquad \mathrm{otherwise}.
\label{collisionkernel2}
\end{eqnarray}
Note that instead of writing down the Kronecker deltas explicitly in
Eq. (\ref{smoluchowsky}) one could include them into the definition
of the collision kernel. In this case, instead of Eqs.
(\ref{collisionkernel1}) and (\ref{collisionkernel2}), we would have
$$\Psi(\ell,m,k,j)=\delta_{\ell+m,k+j}(\delta_{k,0}+\delta_{j,0})/2.$$
This kernel expresses the following process: when two clusters collide a single cluster
merges and it contains all particles of both. The newborn cluster selects its direction of motion,
that could be either clockwise or counterclockwise, with equal
probability. We will refer to (\ref{collisionkernel1}) as the ``random kernel''. Substituting this kernel
in~\eqref{substitution} we find for the generating functions the following differential
system
\begin{eqnarray} \label{ztransform1}
\partial_t F^+(z,t) &=& \frac{1}{2} F^+(z,t)F^-(z,t)-F^+(z,t)F^-(1,t), \\
\partial_t F^-(z,t) &=& \frac{1}{2} F^+(z,t)F^-(z,t) -
F^-(z,t)F^+(1,t).
\label{ztransform2}
\end{eqnarray}
The number of clusters traveling in either direction is given by
\begin{equation*}
N^\pm(t) = \sum_{\ell=1}^\infty f^\pm(\ell,t) = F^\pm(1,t).
\end{equation*}

\begin{lemma}\label{lemmans}
System~(\ref{ztransform1}) and~(\ref{ztransform2}) admits the conserved quantity
$$
N^+(t)-N^-(t),
$$
and if
$$
N^\pm(0) > N^\mp(0) \quad \text{then} \quad N^\mp(t) \to 0 \quad \text{as} \quad t \to \infty.
$$
\end{lemma}

\begin{proof}
The conservation law follows from substituting $z=1$ in~\eqref{ztransform1} and~\eqref{ztransform2} and substracting both equations.
Denoting $N^+(t)-N^-(t)=C_0$ and substituting this quantity back into the differential system we get
\begin{equation}
\partial_t N^-=-\frac{1}{2}N^-(C_0+N^-). \label{logistic}
\end{equation}
This is a Bernoulli equation that can be integrated to yield
\begin{equation}
N^-(t)=\frac{C_0 N^-(0)}{e^{C_0 t/2}N^+(0)-N^-(0)}.
\end{equation}
The conclusion follows immediately.
\end{proof}

Now we propose the following notation for the number of
particles traveling in either direction
\begin{equation*}
M^\pm(t)= \sum_{\ell=1}^\infty \ell f^\pm(\ell,t) =\partial_z F^\pm(1,t),
\end{equation*}
and derive the system
\begin{eqnarray}\label{mplus}
\partial_t M^+ &=& \frac{1}{2} M^-N^+ - \frac{1}{2} M^+N^-, \\
\partial_t M^- &=& \frac{1}{2} M^+N^- - \frac{1}{2} M^-N^+. \label{mminus}
\end{eqnarray}

\begin{lemma}
System~(\ref{ztransform1}) and~(\ref{ztransform2}) admits the conserved quantity
$$
M^+(t) + M^-(t),
$$
and if
$$
N^\pm(0) > N^\mp(0) \quad \text{then} \quad M^\mp(t) \to 0 \quad \text{as} \quad t \to \infty.
$$
\end{lemma}

\begin{proof}
The first conclusion is immediate and follows from adding equations~\eqref{mplus} and~\eqref{mminus}.
Denoting $M^+ + M^- = C_1$, and substituting this quantity back into these equations we get
\begin{equation}
\partial_t M^- = \frac{1}{2} N^- (C_1-M^-)-\frac{1}{2}M^-N^+.
\end{equation}
This is a linear ordinary differential equation that can be integrated to
yield
\begin{eqnarray}
M^-(t) = \hspace{9cm} \\ \nonumber \frac{2C_1N^-(0)^2 + e^{C_0t/2}
[2C_0^2 M^-(0)-2C_1N^-(0)^2+ C_0 C_1 N^-(0) N^+(0)t]}
{2[N^-(0)-e^{C_0t/2}N^+(0)]^2}.
\end{eqnarray}
The second conclusion follows immediately after taking the long time limit in this formula.
\end{proof}

These results indicate the system always becomes ordered
in the long time. In fact, the number of remaining clusters equals
the (absolute value) of the difference of the number of clusters
traveling in either direction initially. So the direction of motion
of the final ordered state is prescribed by the initial condition.
This result should be interpreted probabilistically. The quantities
$N^\pm(t)$ and $M^\pm(t)$ do not represent the dynamics of single
realizations. They are instead the result of averaging these
quantities over many realizations of the stochastic process. So they
represent what happens on average. If the quantity $C_0$ is
macroscopic, i.~e. it is of the same order of magnitude of the
initial total number of clusters, then the result will be observed
in most of the realizations, up to some errors.

\bigskip

The case of the symmetric initial condition $N^+(0)=N^-(0)$ is specially interesting. If this condition holds initially,
then it holds for all times: $N^+(t)=N^-(t) \, \forall t>0$. Furthermore, if we refer to these quantities just as $N(t)$
(precisely because they are equal), we have
\begin{equation}\label{nsymminitcon}
N(t)=\frac{N(0)}{1+tN(0)/2}.
\end{equation}
These results are immediate consequences of the proof of Lemma~\ref{lemmans}.
Equation~\eqref{nsymminitcon} means that clusters in both directions disappear progressively
in the long time. Note that equilibrium is approached exponentially
fast when $C_0 \neq 0$ but the approach becomes algebraic when
$C_0=0$. This case is special because it is the only one, for the
random kernel, for which the kinetic theory does not predict order.
However, if we go beyond the kinetic level and think about the
underlying stochastic process, order should be achieved. This is so
because states in which there exist clusters moving in both
directions are active: collitions are possible and therefore the
system could evolve through collisions towards any other possible
state. On the other hand, ordered states are absorbing: once the
system gets into one of this, collisions are no longer possible, and
therefore there is no possible escape from them. Note the situation
is reminiscent of that of systems undergoing absorbing state phase
transitions, and in particular of the case of two symmetric
absorbing barriers \cite{munoz,lopez}. We devote much of the
remainder of this paper to the study of this case. In section
\ref{randomkernel} we study the kinetic theory corresponding to this
case. We show that the long time asymptotics of the distribution
function adopts a self-similar form in this case. In particular we have
that
$$
f^{\pm}(\ell,t) \to
\frac{1}{t^2}\Phi \left( {\ell \over t} \right) \quad \text{as} \quad t \to \infty,
$$
where $\Phi$ is the self-similar profile,
in a sense that is made totally precise in theorems~\ref{thefun1} and~\ref{thefun2},
which are the main results of that section.
This explains the
fact that the number of clusters approaches zero as time evolves:
clusters of smaller sizes progressively disappear as bigger and
bigger clusters are formed. In section \ref{numericalresults} we
study direct numerical simulations of the aggregation process
precisely in this case. We confirm the theoretical prediction
finding that it is quite probable that order is achieved by means of
the formation of one giant cluster containing all particles in the
system. This fact reflects at the microscopic level the formation of
the self-similar form for long times in the solution to the kinetic
equation. So order is finally achieved in this case too, but it has
a different nature than in the previous situation. Order has the
form of a single giant cluster, not of several clusters; the
direction of motion of the giant cluster is not prescribed at the
initial time, but it is chosen at random with the same probability
among the two different possibilities; and finally order is achieved
after a longer transient in this case.

\subsection{Majority kernel}

We now consider a different kernel which is still integrable. Let us
propose the following Smoluchowski equations
\begin{equation} \label{collisionkernel3}
\partial_t f^\mp(\ell,t) = \sum_{k,j=1}^\infty \delta_{k+j,\ell} \, \frac{k}{k+j}
f^\pm(k,t)f^\mp(j,t)- \sum_{m=1}^\infty f^\pm(m,t) f^\mp (\ell,t),
\end{equation}
which describe collisions among clusters with $k$ and $j$ particles.
After the collision a single cluster composed of all involved
particles merges. The resulting cluster has a probability $k/(k+j)$
of traveling in the same direction the cluster with $k$ particles
was traveling, and a probability $j/(k+j)$ of traveling in the
direction the cluster with $j$ particles was traveling. We will
refer to this as the ``majority kernel''.
Our approach will be the same as in the other case and starts
with equation~\eqref{substitution}.
For the generating
functions we find
\begin{eqnarray}\label{majority1}
\partial_t (\partial_z F^+) &=& F^- (\partial_z F^+) - (\partial_z F^+) N^-,
\\
\partial_t (\partial_z F^-) &=& F^+ (\partial_z F^-) - (\partial_z F^-) N^+,
\label{majority2}
\end{eqnarray}
where the number of clusters $N^\pm (t)= F^\pm(1,t)$. Denoting
$M^\pm (t)=
\partial_z F^\pm(1,t)$ the number of particles, we find
\begin{equation}\label{noorder}
\partial_t M^+=\partial_t M^-=0,
\end{equation}
so the initial number of particles traveling in either direction
does not change over time. This result apparently means the system
could never become ordered. However, the actual meaning is that the
number of particles moving in either direction, averaged over many
realizations of the stochastic process, is constant. As we have
already said the kinetic description implicitly assumes this
average. In other words, in the dynamics dictated by this kernel the
individual collisions do not conserve momentum, but momentum is
conserved on average. In any case, individual realizations should
become ordered, because as in the former case, order is an absorbing
state. In particular, the situation is analogous to the previous
one: we have a system with two symmetric absorbing barriers again.
Result (\ref{noorder}) could indicate the advent of dynamic
self-similarity, just like in the former case, in the long time
limit. Although for the majority kernel we did not develop the
theory to the same extent as in the former case, we have built some
progress which is detailed in section~\ref{majoritykernel}. In this
section we show (computational) evidence of the existence of
self-similar asymptotic behavior, although we do not clarify the
conditions under which it develops.

\section{Self-similar asymptotics for the random kernel}
\label{randomkernel}

The goal of the present section is to carry out the rigorous analysis of the asymptotic behavior of our
kinetic theory in the case that the random kernel is considered. We will divide the analysis in several sections.
In section~\ref{pre} we set the preliminary results for the analysis of a simplified
situation in which we assume that all moments of both initial conditions are
identical. In~\ref{formalsymm} we formally analyze the time asympotics of our model in this simplified case.
In~\ref{asymm} we calculate the exact solution of the model in the full case. The formal asymptotic analysis of this case
is presented in~\ref{formalasymm}. Finally, rigorous convergence results are set forth in sections~\ref{convergence} and~\ref{sharp}.
Let us remark that all the sections except~\ref{formalsymm} and~\ref{formalasymm} present rigorous results and that, furthermore,
the formal results in~\ref{formalsymm} and~\ref{formalasymm} are rigorously justified in~\ref{convergence} and~\ref{sharp}.

\subsection{Preliminaries}\label{pre}

As we have already mentioned, we will now
concentrate on the case of a symmetric initial condition
($N^+(0)=N^-(0)$). Our goal is to derive its self-similar asymptotic behavior.
Before moving to the general case, we start with a simplified
situation: we assume that all moments of both initial condition are
identical. This translates in considering equations
(\ref{ztransform1}) and (\ref{ztransform2}) with initial conditions
$F^+(z,0)=F^-(z,0)$ and in particular $N^+(0)=N^-(0)$.

The probabilistic approach employed in the previous section
implies the following structure of the initial condition:
\begin{equation}\label{initcon}
F\left( z,0\right) =F_{0}\left( z\right) =\sum_{n=0}^{\infty
}a_{n}z^{n}\;\;,\;\;a_{n}\geq 0.
\end{equation}
We may assume $N\left( 0\right) =\sum_{n=0}^{\infty }a_{n}=1$
without loss of generality; in particular this implies $F_0(z)$ is
holomorphic on the unit disc in $\mathbb{C}$ as a function of $z$. From now
on we will consider $z$ as complex variable taking its values on the unit disc of $\mathbb{C}$
and a real $0 \le t < \infty$.

We start
proving the following
\begin{lemma}\label{lemmasymm}
Consider the system of differential
equations~(\ref{ztransform1})-(\ref{ztransform2}) subject to the
initial condition $F^+(z,0)=F^-(z,0)$. Then we have
$F^+(z,t)=F^-(z,t)$ during the lapse of existence of the solutions.
\end{lemma}

\begin{proof}
Subtracting both equations one finds
\begin{equation}
\partial_t(F^+-F^-)=-(F^+-F^-)N^- + C_0 F^-.
\end{equation}
Using $N^+(0)=F^+(1,0)=F^-(1,0)=N^-(0)$ we find $C_0=0$ and
consequently
\begin{equation}
\partial_t(F^+-F^-)=-(F^+-F^-)N^-(t).
\end{equation}
This equation can be integrated to yield
\begin{equation}
F^+(z,t)-F^-(z,t)=[F^+(z,0)-F^-(z,0)]\exp \left[ -\int_0^t
N^-(t')dt' \right]=0,
\end{equation}
so we conclude.
\end{proof}

\begin{definition}
We denote $F=F\left( z,t\right) \equiv F^+(z,t) = F^-(z,t)$ and
$N=N\left( t\right) =F\left( 1,t\right) \equiv N^+(t)=N^-(t)$
whenever these equalities hold.
\end{definition}

\begin{corollary}
The function $F$ obeys the following differential equation:
\begin{equation*}
\partial_t F=\frac{1}{2}F^{2}-NF.
\end{equation*}
\end{corollary}

\begin{lemma}\label{expsol}
Lets assume $F(z,0)=F_0(z)$ and $N\left( 0\right) =1$. Then
\begin{equation*}
F\left( z,t\right) =\frac{1}{\left( 1+\frac{t}{2}\right) ^{2}}\frac{%
F_{0}\left( z\right) }{\left[ 1-\frac{t}{\left( 2+t\right)
}F_{0}\left( z\right) \right] }.
\end{equation*}
\end{lemma}

\begin{proof}
Choosing $z=1$ we obtain
\begin{equation*}
\partial_t N = -\frac{N^{2}}{2}\;\;,
\end{equation*}
and so
\begin{equation*}
N\left( t\right) =\frac{1}{1+\frac{t}{2}}.
\end{equation*}
Substituting this value of $N(t)$ in the equation for $F$ we
reduce this second equation to an ODE of Bernoulli type. The
solution of the initial value problem for $F$ can be therefore
found by means of the change of variables
\begin{equation*}
H = \frac{1}{F}.
\end{equation*}
We find
\begin{equation*}
H\left( z,t\right) = \left( 1+\frac{t}{2}\right) ^{2}\left[
\frac{1}{F_{0}\left( z\right) }-\frac{t}{\left( 2+t\right)
}\right]. \\
\end{equation*}
And the desired solution for $F(z,t)$ follows from
$F(z,t)=1/H(z,t)$ whenever $H(z,t) \neq 0$.
Since $\left|
F_{0}\left( z\right) \right| \leq 1$ for $|z| \leq 1$, as can be
deduced from the Taylor series in~\eqref{initcon} by means of the triangle inequality,
$H(z,t) \neq 0$ for every $z$ in the unit disc of $\mathbb{C}$ and $0 \le t < \infty$, so we conclude.
\end{proof}

\begin{corollary}
The result in Lemma~\ref{lemmasymm} holds for all time $t>0$.
\end{corollary}

\bigskip

We have already seen in the proof of Lemma~\ref{expsol} that
$\left|F_{0}\left( z\right) \right| \leq 1$ for $|z| \leq 1$, as can be
inferred from the triangle inequality applied to the Taylor series~\eqref{initcon}. We also have the
strict inequality
\begin{equation*}
\left| F_{0}\left( z\right) \right| < 1\text{ if \ \ }\left|
z\right| < 1\;,
\end{equation*}
if we consider the strict triangle inequality instead.

By assumption we know $F_0(1)=1$; on the other hand we can have
$F_{0}\left( z\right) =1$ for $\left| z\right| =1$ at some points
$z \neq 1.$ In particular we have the following result:

\begin{proposition}\label{bezout}
Let $F_0(z)$ be as explained above. Let $d=\mathrm{gcd}(n|\{a_n \neq
0\}), \, n \in \mathbb{N}^+$. Then $F_0(z)=1$ if and only if
$z^d=1$.
\end{proposition}

\begin{proof}
Suppose that $a_{n}\neq 0$ only for $%
n=md ,\;m=1,2,\dots,\;d >1$. Then
\begin{equation*}
F_{0}\left( z\right) =\sum_{m=1}^{\infty }a_{md }\left( z^{d
}\right) ^{m}.
\end{equation*}
It then follows that $F_{0}\left( z\right) =1$ for $z^{d }=1.$

Now we prove the reciprocal. If $a_n=0$ except for a subsequence
then
$$
F_0(z)=\sum_{n=1}^\infty a_n z^n =\sum_{m=1}^\infty
a_{md}z^{md},
$$
where $a_{md}$ contains the non-vanishing subsequence and
$d$ is the greatest common divisor of all the positive
coefficients $a_n$: $d=\mathrm{gcd}(n|\{a_n \neq 0\})$,
$n=1,2,\cdots$.

In order to have $F_0(z)=1$ we need
$$z^{md}=1 \,\, \forall \, m \in \mathbb{N}^+, $$
otherwise the strict triangular inequality applied to the Taylor
series implies $|F_0(z)|<1$. So we have, by the definition of
greatest common divisor, the existence of a pair of relative
primes $m_1,m_2>0$ such that $z^{m_1d}=1$ and $z^{m_2d}=1$.
This automatically implies
$$
(z^d)^{xm_1+ym_2}=1,
$$
for any two integers $x$ and $y$. A direct consequence of B\'ezout
lemma is the existence of pairs of integers $(x,y)$ satisfying the
diophantine equation $xm_1+ym_2=1$; this leads to the desired
conclusion $z^d=1$.
\end{proof}

\begin{remark}\label{bezoutrem}
The previous result immediately implies that the expression
$F_0(z)-1$ has a family of zeros at the $d -$roots of unity.
\end{remark}

\bigskip

This property in turn assures the existence of a {\it
``characteristic wavelength''} in the following sense: If the
initial condition is composed only by clusters having a number of
particles which is multiple of some integer then only clusters
with a number of particles multiple of the same integer will be
generated.

\bigskip

The computations in the following section will assume that $z=1$ is the only root of
$F_{0}\left( z\right) =1$. If there were more roots the
computations would be analogous.

\bigskip

We have $\left| 1-\frac{t}{\left( 2+t\right) }F_{0}\left( z\right)
\right| >0 $ if \thinspace $z\neq 1,\;\left| z\right| \leq 1.$
This implies $\left| F\left( z,t\right) \right| \leq
\frac{C}{t^{2}}$ if $t\rightarrow \infty $ for every $z\neq 1,$
$\left| z\right| \leq 1$, and for a suitable constant $C$ (the
symbol $C$ will be used to denote a generic constant whose value
may change from line to line). On the contrary $\left| F\left(
z,t\right) \right| \sim \frac{1}{t}$ when $t\rightarrow \infty
$ if $z=1$.

\bigskip

In the following we will obtain the asymptotic behavior of the
solutions in this case. Note that an immediate corollary of
Lemma~\ref{lemmasymm} is that $f^+(\ell,t)=f^-(\ell,t) \, \forall \, t \ge 0$;
therefore we will use the following

\begin{definition}
We denote $f(\ell,t) \equiv f^+(\ell,t)=f^-(\ell,t)$
whenever this equality holds.
\end{definition}

\subsection{Formal asymptotics for the symmetric case}\label{formalsymm}

In this section we will carry out a formal analysis of the
asymptotic behavior of $f(\ell,t)$; these results will be rigourously
justified in the following sections.
Using the Cauchy formulas for the Taylor coefficients we find
\begin{equation*}
f\left( \ell ,t\right) =\frac{1}{2\pi i}\int_{\left| z\right| =1}\frac{%
F\left( z,t\right) }{z^{\ell +1}}dz
\end{equation*}
where the line integration is carried out in the counterclockwise
direction. Then we have
\begin{equation*}
f\left( \ell ,t\right) =\frac{1}{2\pi i}\frac{1}{\left(
1+\frac{t}{2}\right)
^{2}}\int_{\left| z\right| =1}\frac{F_{0}\left( z\right) }{\left[ 1-\frac{t}{%
\left( 2+t\right) }F_{0}\left( z\right) \right] }\frac{dz}{z^{\ell
+1}}.
\end{equation*}
It is easy to see that the integral is well defined for all finite
$t$.

\bigskip

We start our analysis with the following assumption on the initial
data
$$f\left( \ell ,0\right) \sim e^{-C \ell},$$
for an arbitrary positive constant $C$. This assumption is purely
technical: we will use it in the first instance to simplify the
analysis, but we will substitute it for a more general one further
below. In this case we have $F_{0}\left( z\right)$ is holomorphic
in $\left| z\right| <1+\delta $, $\delta >0$ small enough, and
$0<F_{0}^{\prime }\left( 1\right) =\sum_{\ell =1}^{\infty }\ell
f\left( \ell ,0\right) <\infty .$

It is interesting to examine the integrand pole situated next to
$z=1.$ Such a pole is situated at the root of
\begin{equation*}
1-\frac{t}{\left( 2+t\right) }F_{0}\left( z_{t}\right) \approx 0.
\end{equation*}

If $t=\infty $ we know that the root is simple since
$0<F_{0}^{\prime }\left( 1\right) $ and $z_{\infty }=1.$ The
implicit function theorem assures the existence of one root for
$t$ large enough and one has the approximation
\begin{equation*}
\frac{2}{\left( 2+t\right) }-\frac{t}{\left( 2+t\right)
}F_{0}^{\prime }\left( 1\right) \left( z_{t}-1\right) =0
\end{equation*}
from where
\begin{equation*} z_{t} \approx 1+\frac{2}{t}
\frac{1}{F_{0}^{\prime }\left( 1\right) }\;\;\text{when
}t\rightarrow \infty.
\end{equation*}

Deforming contours and applying the residue theorem we obtain
\begin{eqnarray*}
f\left( \ell ,t\right) &=&-\frac{1}{\left( 1+\frac{t}{2}\right) ^{2}}\frac{%
F_{0}\left( z_{t}\right) }{\left( z_{t}\right) ^{\ell +1}}\cdot
\lim_{z\rightarrow z_{t}}\frac{\left( z-z_{t}\right) }{\left( 1-\frac{t}{%
\left( 2+t\right) }F_{0}\left( z\right) \right) }+ \\
&&+\frac{1}{2\pi i}\frac{1}{\left( 1+\frac{t}{2}\right)
^{2}}\int_{\left|
z\right| =1+\frac{\delta }{2}}\frac{F_{0}\left( z\right) }{\left[ 1-\frac{t}{%
\left( 2+t\right) }F_{0}\left( z\right) \right] }\frac{dz}{z^{\ell
+1}}
\end{eqnarray*}
where $\delta $ is chosen small enough in order to avoid
additional singularities in the integrand. By employing
L'H\^{o}pital method
\begin{equation*}
\lim_{z\rightarrow z_{t}}\frac{\left( z-z_{t}\right) }{\left( 1-\frac{t}{%
\left( 2+t\right) }F_{0}\left( z\right) \right) }=-\frac{2+t}{t}\frac{1}{%
F_{0}^{\prime }\left( z_{t}\right) }
\end{equation*}
from where
\begin{eqnarray*}
f\left( \ell ,t\right) &=&\frac{1}{\left( 1+\frac{t}{2}\right) ^{2}}\frac{1}{%
\left( z_{t}\right) ^{\ell +1}}\left( \frac{2+t}{t}\right)
\frac{F_{0}\left(
z_{t}\right) }{F_{0}^{\prime }\left( z_{t}\right) }+ \\
&&+\frac{1}{2\pi i}\frac{1}{\left( 1+\frac{t}{2}\right)
^{2}}\int_{\left|
z\right| =1+\frac{\delta }{2}}\frac{F_{0}\left( z\right) }{\left[ 1-\frac{t}{%
\left( 2+t\right) }F_{0}\left( z\right) \right] }\frac{dz}{z^{\ell
+1}}.
\end{eqnarray*}

This gives the self-similar behavior of the solution when
$t\rightarrow \infty ,$ since
\begin{equation*}
\frac{1}{\left( 1+\frac{t}{2}\right) ^{2}}\frac{1}{\left(
z_{t}\right)
^{\ell +1}}\left( \frac{2+t}{t}\right) \frac{F_{0}\left( z_{t}\right) }{%
F_{0}^{\prime }\left( z_{t}\right) }\approx \frac{4}{t^{2}}\frac{1}{%
F_{0}^{\prime }\left( 1\right) }\frac{1}{\left( z_{t}\right)
^{\ell +1}}.
\end{equation*}

In particular, in the region $\ell $ of order $t$ one obtains
\begin{equation*}
\frac{1}{\left( z_{t}\right) ^{\ell +1}} \approx \frac{1}{\left( 1+\frac{2}{%
t}\frac{1}{F_{0}^{\prime }\left( 1\right) }\right) ^{\ell +1}}
\approx \exp \left( -\frac{2 \ell }{ t F_{0}^{\prime }\left(
1\right) }\right).
\end{equation*}
So we obtain the self-similar structure $\frac{1}{t^{2}}\Phi \left( \frac{\ell }{t}%
\right)$ of the solution
\begin{equation}
f\left( \ell ,t\right) \approx \frac{4}{t^{2}}\frac{1}{%
F_{0}^{\prime }\left( 1\right) } \exp \left( -\frac{2 \ell}{t
F_{0}^{\prime }\left( 1\right) }\right)
\end{equation}
for long $t$, large $\ell$, and both sharing the same magnitude.
We note that $2F_{0}^{\prime} \left( 1 \right) = M^+(0) + M^-(0) =: \ell_0$, the
total number of particles in the system (note this is a conserved quantity as the
collisions we consider are mass conserving).

On the other hand, it is interesting to observe that in the region where
$\ell $ is of order one the resulting integral term in the
$f\left( \ell ,t\right) $ formula yields an order $\frac{1}{t^{2}}$ term
depending on the values of $F_{0}$ in regions where $z$ is not in
the neighborhood of one. The contribution to the mass from such a
term approaches zero when $\ell \rightarrow \infty$, since this
contribution is relevant only if $\ell$ is of order one. With the
hypothesis of analyticity made we have that such a term is bounded
by $\frac{e^{-C\ell }}{t^{2}}$, except for a multiplicative
constant, when $\ell \rightarrow \infty$. Anyway this term
induces a sort of {\it ``boundary layer''} in the sense that the
region where $\ell$ is of order one cannot be described using the
self-similar function. Rigorous convergence results will be shown
further below, where we will precisely state how the solution of
the coagulation equation converges to the self-similar function.
Also, both contributions to the dynamics, self-similar function
and boundary layer, will be clearly identified in our numerical
simulations in section \ref{numericalresults}.

\subsection{The asymmetric case}\label{asymm}

Now we move to calculate the solution of the random kernel model
when $N^+(0)=N^-(0)$ (and consequently $N^+(t)=N^-(t)$) but
$F^+(z,0) \neq F^-(z,0)$ for $z \neq 1$.

Suppose $N^{+}=N^{-}=N$; denote
\begin{equation*}
G\left( z,t\right) :=F^{+}\left( z,t\right) -F^{-}\left(
z,t\right) \text{;}
\end{equation*}
we have, from equations (\ref{ztransform1}), (\ref{ztransform2})
and (\ref {logistic})
\begin{equation*}
\partial_{t}G=-NG,\qquad\partial_{t}N=-\frac{1}{2}N^{2}\text{.}
\end{equation*}
The second equation may be integrated to give
\begin{equation*}
N\left( t\right) =\frac{N_{0}}{1+\frac{N_{0}}{2}t},\qquad
N_{0}:=N\left( 0\right)
\end{equation*}
which results in
\begin{equation*}
G\left( z,t\right) =\frac{1}{\left( 1+\frac{N_{0}}{2}t\right) ^{2}}%
G_{0}\left( z\right) ,\qquad G_{0}\left( z\right) :=G\left(
z,0\right) .
\end{equation*}
Replacing $F^{+}=G+F^{-}$ in equation (\ref{ztransform2}) we get
\begin{equation}
\partial_{t}F^{-}=\frac{1}{2}\left( F^{-}\right) ^{2}+\frac{1}{2}%
GF^{-}-NF^{-}. \label{efemenos}
\end{equation}
This equation of Bernoulli type can be explicitly solved making the
substitution $H^{-}=1/F^{-}$. It turns out that
\begin{equation}
\partial_{t}H^{-}=-\frac{1}{2}-\frac{G}{2}H^{-}+NH^{-}. \label{riccati}
\end{equation}
Equation (\ref{riccati}) can be integrated to give
\begin{eqnarray}
\nonumber
H^{-}\left( z,t\right) =\left( 1+\frac{N_{0}}{2}t\right) ^{2}\left\{ \frac{1%
}{G_{0}\left( z\right) }\left[ \exp\left( -\frac{G_{0}\left( z\right) }{2}%
\frac{t}{1+\frac{N_{0}}{2}t}\right) -1\right] \right. \\ \left.
+\frac{1}{F^{-}\left(
z,0\right) }\exp\left( -\frac{G_{0}\left( z\right) }{2}\frac{t}{1+\frac{N_{0}%
}{2}t}\right) \right\}. \nonumber
\end{eqnarray}
In particular, the solution $F^{-}$ to (\ref{efemenos}) has the
simple expression
\begin{eqnarray}
\nonumber
& & F^{-}\left( z,t\right) \\ \nonumber
&=& \frac{1}{\left( 1+\frac{N_{0}}{2}t\right) ^{2}} \times \hspace{8cm} \\ & & \frac{%
1}{\dfrac{1}{G_{0}\left( z\right) }\left[ \exp\left( -\dfrac
{G_{0}\left(
z\right) }{2}\dfrac{t}{1+\frac{N_{0}}{2}t}\right) -1\right] +\dfrac{1}{%
F^{-}\left( z,0\right) }\exp\left( -\dfrac{G_{0}\left( z\right) }{2}\dfrac{t%
}{1+\frac{N_{0}}{2}t}\right) } \nonumber
\end{eqnarray}
where, we recall, we have denoted $G_{0}\left( z\right)
=F^{+}\left( z,0\right) -F^{-}\left( z,0\right) $.

\bigskip

Our goal is to prove rigorous estimates for the asymptotic temporal
behavior of the distribution functions $f^\pm(\ell,t)$. The analysis will be split
in the following sections. We start by proving the
necessary estimates in the complex plane for the corresponding
generating functions. We proceed analogously to what we did in the
previous case of a totally symmetric initial condition. This is, we
start looking for poles in the integrand of the integral expression
for $f^-(\ell,t)$.

The main objective of this section is to prove that the exact solution for
$F^{-}\left( z,t\right) $ is given by
\begin{equation}
F^{-}\left( z,t\right) =\frac{1}{\left( 1+\frac{N_{0}t}{2}\right) ^{2}}\frac{%
1}{\frac{1}{G_{0}\left( z\right) }\left[ \exp \left(
-\frac{G_{0}\left( z\right)
}{2}\frac{t}{1+\frac{N_{0}t}{2}}\right) -1\right] +\frac{\exp
\left( -\frac{G_{0}\left( z\right) }{2}\frac{t}{1+\frac{N_{0}t}{2}}\right) }{%
F^{-}\left( z,0\right) }} \label{fminus}
\end{equation}
where $G=F^{+}-F^{-}$. First we reformulate the problem in purely
analytic terms. We
can rescale the factor $N_{0}$%
\begin{equation*}
g_{0}\left( z\right) =\frac{G_{0}\left( z\right) }{N_{0}}\;\;,\;\;f_{0}%
\left( z\right) =\frac{F^{-}\left( z,0\right) }{N_{0}}.
\end{equation*}
It is convenient to rewrite the solution of $F^{-}$ to check the
condition that avoids singularities in the present case. We have
\begin{equation*}
F^{-}\left( z,t\right) =\frac{1}{\left( 1+\frac{N_{0}t}{2}\right) ^{2}}%
\frac{\exp \left( g_{0}\left( z\right) \frac{t}{\frac{2}{N_{0}}+t}\right) }{%
\left[ \frac{\left[ 1-\exp \left( g_{0}\left( z\right) \frac{t}{\frac{2}{%
N_{0}}+t}\right) \right] }{g_{0}\left( z\right)
}+\frac{1}{f_{0}\left( z\right) }\right] },
\end{equation*}
where we have also rescaled $F^{-}\left( z,t\right) \to F^{-}\left(
z,t\right) / N_0$. Therefore the singularities arise at the zeroes
of
\begin{equation*}
\left[ \frac{\left[ 1-\exp \left( g_{0}\left( z\right) \frac{t}{\frac{2}{%
N_{0}}+t}\right) \right] }{g_{0}\left( z\right)
}+\frac{1}{f_{0}\left( z\right) }\right].
\end{equation*}
This function can have poles at the zeroes of $f_{0}\left(
z\right) .$ At such points $F^{-}\left( z,t\right) =0.$ Notice
that the zeroes of $g_{0}$ are not problematic due to the cancellations of the zeroes of $\frac{%
\left[ 1-\exp \left( g_{0}\left( z\right) \frac{t}{\frac{2}{N_{0}}+t}\right) %
\right] }{g_{0}\left( z\right) }$ at the numerator and
denominator.

\begin{proposition}
Consider the above expression for $F^-(z,t)$. This expression has
no poles neither for $t \in [0,\infty)$ and $|z| \le 1$ nor for
$t=\infty$ and $|z| <1$.
\end{proposition}

\begin{proof}
Expression~(\ref{fminus}) can be bounded using the fact
$|F^-(z,t)| \le F^-(1,t)$ for $|z| \le 1$; this inequality is
obtained using the triangular inequality on the power series which
defines this quantity, exactly as we did it in the previous
section. We have also shown the coagulation equation implies that the number of particles $%
N^{-}\left( t\right) $ is bounded and decreases as $\frac{1}{t}$ for large $%
t.$ Then
\begin{equation}
\left| F^{-}\left( z,t\right) \right| \leq
\frac{C}{1+t}\;\;,\;\;\left| z\right| \leq 1. \label{E1}
\end{equation}
This shows no poles are present in this expression for finite $t$
and $|z| \le 1$.

Now we move to the case $t=\infty$. In this case our problem
reduces to prove that the following function
\begin{equation}
\frac{1}{g_{0}\left( z\right) }\left[ \exp \left( -g_{0}\left(
z\right) \right) -1\right] +\frac{\exp \left( -g_{0}\left(
z\right) \right) }{f_{0}\left( z,0\right) } \label{nozeros}
\end{equation}
does not have any zero in the disk $\left| z\right| <1.$ We
proceed by contradiction. We assume there is a zero for this
expression at $z_0$ such that $|z_0| <1$. For long $t$ we have the
approximation
\begin{eqnarray} \nonumber
& & \frac{1}{g_{0}\left( z\right) }\left[ \exp \left(
-\frac{g_{0}\left( z\right)
}{2}\frac{t}{\frac{1}{N_0}+\frac{t}{2}}\right) -1\right] +
\frac{\exp
\left( -\frac{g_{0}\left( z\right) }{2}\frac{t}{\frac{1}{N_0}+\frac{t}{2}}\right) }{%
f_{0}\left( z,0\right) } \\ &=&
\nonumber \frac{1}{g_{0}\left( z\right) }\left[ \exp \left(
-g_{0}\left( z\right) \right) -1\right] + \frac{\exp \left(
-g_{0}\left( z\right) \right) }{f_{0}\left(
z \right) } + \\ & &
2 \frac{\exp \left( -\frac{g_{0}\left( z\right)
}{N_{0}}\right)\left[ f_{0}\left( z \right) + g_{0}\left( z\right)
\right]}{f_{0}\left( z \right) N_0^2} \frac{1}{t} + O \left(
\frac{1}{t^2} \right). \label{nozeros2}
\end{eqnarray}
Due to Rouche's Theorem, expressions~(\ref{nozeros})
and~(\ref{nozeros2}) should have the same number of zeros inside the
disk $|z|=1$ for a sufficiently large $t$, provided $F^{-}\left(
z,0\right)$ has no zeros on $|z|=1$. If this number is nonzero this
implies that $F^{-}\left( z,t\right) $ blows up for a finite value
of $t$ at the interior of the disk $\left| z\right| =1$ but this
contradicts the estimate (\ref{E1}), so we conclude in this case. If
$F^{-}\left( z_1,0\right)=0$ for some $z_1$ on $|z|=1$ then we know
this zero is isolated because the function is holomorphic. We may
know apply Rouche's Theorem on the contour $|z|=1-\epsilon$; we can
make $\epsilon$ arbitrarily small so we have $|z'|< 1-\epsilon$ for
any possible pole $z'$ in the open disc $|z|<1$, and by enlarging
the time $t$ we may keep under control the $O(t^{-1})$ term
in~(\ref{nozeros2}). This leads to the desired contradiction in the
general case.
\end{proof}

\bigskip

\begin{corollary}
Formula~\eqref{fminus} is the actual solution to equation~\eqref{efemenos} for any $|z| \le 1$ and $0 \le t <\infty$.
\end{corollary}

\begin{remark}
Notice that there is at least a zero of (\ref{nozeros}) at the point
$z=1$.
\end{remark}

\begin{remark}
In order to obtain convenient representation formulas for the
self-similar asymptotics of the coagulation equation under study it
is necessary to locate the remaining zeros of (\ref{nozeros}) in
$\left\{ \left| z\right| =1\right\} \setminus \left\{ 1\right\}$. As
in the case of the totally symmetric initial conditions it is
possible to have more than one zero at the unit disk $\left|
z\right| =1$. For this to happen, by Proposition~\ref{bezout}, we
have the following necessary and sufficient condition
\begin{equation*}
F_{0}^{\pm}(z)=\sum_{m=1}^{\infty }a_{m d}^{\pm }\left( z^{d}\right)
^{m}, \qquad d >1.
\end{equation*}
In this case there are several zeros $z^{d}=1$. Physically the
resulting solution will have gaps of length $d$. In summary, there
is a complete analogy with what happened in the case of the totally
symmetric initial condition, see the discussion in section~\ref{pre}.
\end{remark}

\bigskip

In order to find the roots of the equation
\begin{equation}
\label{roots} \frac{1-e^{g_0(z)}}{g_0(z)}+\frac{1}{f_0(z)}=0,
\end{equation}
we reformulate the problem in terms of the function $h_0=f_0+g_0$,
what yields
$$
\frac{e^{f_0(z)}}{f_0(z)}=\frac{e^{h_0(z)}}{h_0(z)},
$$
after assuming that $g_0(z) \neq 0$ for every $z$ that solves the
equation. We have also used the fact that both $f_0$ and $g_0$ are
bounded and the equality holds for $f_0=0$ only if $g_0=0$ and vice
versa. We note the points $z$ such that $g_0(z)=0$ are not important
for our present purposes since at these points the equality holds
only if $f_0(z)=1$, what in turn implies $h_0(z)=1$. From now on we
drop the subindex ``$0$'' in order to simplify the notation of these
functions.

\bigskip

\begin{lemma}\label{lemtreefunc}
Let $f(z)=\sum_{n=1}^\infty a_n z^n$ and $h(z)=\sum_{n=1}^\infty b_n
z^n$ be two holomorphic functions defined for $z \in \mathbb{C}$ on
the closed disc $|z| \le 1$. Assume $a_n,b_n \in \mathbb{R}^+ \cup
\{0\}$ and $f(1)=h(1)=1$. Then a complex number $z_0$ such that
$|z_0| \le 1$ is a solution to the equation
$$
\frac{e^{f(z)}}{f(z)}=\frac{e^{h(z)}}{h(z)}
$$
if and only if $f(z_0)=h(z_0)$.
\end{lemma}

\begin{proof} The direct implication is obvious. So we will
concentrate on proving the inverse implication in the following.

We can cast the equation into the following form
\begin{equation}
\label{zeroes} e^{-f(z)}f(z)=e^{-h(z)}h(z),
\end{equation}
using the fact that both $f$ and $h$ are bounded and the original
equality holds for $f=0$ only if $h=0$ and vice versa.

Note that both $f$ and $h$ fulfill the inequality
$$
|f| \le 1, \qquad |h| \le 1, \qquad \mathrm{for} \qquad |z| \le 1,
$$
so in particular these functions map the closed disc $|z| \le 1$
onto itself.

The complex function $w \to we^{-w}$, where $\{w \in \mathbb{C} \,
\left| \,\,\,\, |w| = 1 \} \right.$, has winding number $1$. This
can be seen by noting that $w=e^{i\theta}$, $\theta \in [0,2\pi)$,
and so we have
$$
\theta \to e^{-\cos(\theta)} e^{i[\theta-\sin(\theta)]},
$$
and so the phase $\theta-\sin(\theta)$ clearly reveals that this
function winding number is $1$. By invoking the Argument Principle
we may conclude that the equation
$$
w \, e^{-w}=w_0
$$
has at most one solution for $w$ in the disc $|w|<1$ for some
fixed $w_0$. Now let us focus on the $|w|=1$ case. In this case we
see that the phase $\theta-\sin(\theta) \in [0,2\pi)$ and it is
strictly increasing when $\theta \in [0,2\pi)$. This tells us that
the mapping $w \to we^{-w}$ is univalued in $|w|=1$ as well,
yielding us the desired conclusion.
\end{proof}

\begin{remark}
The same conclusion could be derived using the properties of a
well known special function. Consider now the equation
$$
\mathcal{T}(z) \, e^{-\mathcal{T}(z)}=z,
$$
that defines the \emph{tree function} $\mathcal{T}=\mathcal{T}(z)$
\cite{tree}. Alternatively we have
$$
\mathcal{T}(z)=-\mathcal{W}(-z),
$$
where $\mathcal{W}=\mathcal{W}(z)$ is the Lambert Omega function
\cite{omega}. The tree function is multivalued (as the Lambert
function is), but we know, by using the properties of the Lambert
function proved in \cite{omega}, that for $|\mathcal{T}| \le 1$,
$\mathcal{T}(z)$ maps bijectively onto the $z-$plane, what
guarantees the desired result.
\end{remark}

\bigskip

\begin{corollary}
Let $f(z)$ and $g(z)$ be as above. Then $z_0$ is a solution to Eq.
(\ref{roots}) if and only if $z_0^d=1$, where
$d=\mathrm{gcd}(n|\{a_n \neq 0\} \wedge \{b_n \neq 0\}) \, \forall
\, n \in \mathbb{N}^+$.
\end{corollary}

\begin{proof}
From Lemma \ref{lemtreefunc} we find that the only possible
solutions to Eq. (\ref{roots}) are those fulfilling $g(z_0)=0$. This
in turn forces $f(z_0)=1$. This equality only holds when
$z^{\ell_1}=1$, where $\ell_1=\mathrm{gcd}(n|\{a_n \neq 0\})$, as
can be directly deduced from Proposition \ref{bezout} and Remark
\ref{bezoutrem}. Another consequence of Lemma \ref{lemtreefunc} is
that $f(z_0)=h(z_0)$ must hold simultaneously to $f(z_0)=1$, so the
desired conclusion follows.
\end{proof}

\subsection{Formal asymptotics for the asymmetric case}\label{formalasymm}

We now proceed to calculate the self-similar asymptotic form of the
solution to the problem with identical initial conditions
$F^+(z,0)=F^-(z,0)$. We will assume $\ell/d \in \mathbb{N}$ because
otherwise $f(\ell,t) \equiv 0$. We already found the following
explicit formula for the solution
$$ f(\ell,t)=\frac{1}{2\pi i}
\frac{1}{\left( 1+{t \over 2} \right)^2}
\int_{|z|=1}\frac{F_0(z)}{1-\frac{t}{(2+t)}F_0(z)}
\frac{dz}{z^{\ell+1}},
$$
where $F_0(z)$ is the initial condition of the generating
function. We know that this function can be written as
$F_0(z)=Q(z^d)$, where $d$ is some positive integer and $Q(z^d)$
is a function such that $Q(1)=1$ and $Q(w) \neq 1$ for $w \in
\mathbb{C}$, $|w| \le 1$, and $w \neq 1$. In this case we have
$$
f(\ell,t)=\frac{1}{2\pi i} \frac{1}{\left( 1+{t \over 2}
\right)^2} \int_{|z|=1}\frac{Q(z^d)}{1-\frac{t}{(2+t)}Q(z^d)}
\frac{dz}{z^{\ell+1}},
$$
and changing variables $z^d=\zeta$ we obtain
$$
f(\ell,t)=\frac{1}{2\pi i d} \frac{1}{\left( 1+{t \over 2}
\right)^2}
\int_{|\zeta|=1}\frac{Q(\zeta)}{1-\frac{t}{(2+t)}Q(\zeta)}
\frac{d\zeta}{\zeta^{1+\ell/d}}.
$$
There is a pole in the integrand next to $\zeta=1$:
$$
\zeta_t = 1 + \frac{2}{t} \frac{1}{Q'(1)} \quad + \quad
\mathrm{higher} \quad \mathrm{order} \quad \mathrm{terms},
$$
so we may write
\begin{eqnarray}\nonumber
f(\ell,t)&=&\frac{1}{2\pi i d} \frac{1}{\left(
1+{t \over 2} \right)^2} \int_{|\zeta|=1+{\delta \over
2}}\frac{Q(\zeta)}{1-\frac{t}{(2+t)}Q(\zeta)}
\frac{d\zeta}{\zeta^{1+\ell/d}} \\ \nonumber & &
-\frac{1}{d \left( 1+{t \over
2} \right)^2} \frac{Q(\zeta_t)}{\zeta_t^{1+\ell/d}} \lim_{\zeta
\to \zeta_t}\frac{\zeta-\zeta_t}{1-{t \over (t+2)}Q(\zeta)}.
\end{eqnarray}
The limit can be calculated with L'H\^{o}pital rule
$$
\lim_{\zeta \to \zeta_t} \frac{\zeta-\zeta_t}{1-{t \over
2+t}Q(\zeta)}=-\frac{2+t}{t Q'(\zeta_t)},
$$
and so we have
$$
-\frac{1}{d \left( 1+{t \over 2} \right)^2}
\frac{Q(\zeta_t)}{\zeta_t^{1+\ell/d}} \lim_{\zeta \to
\zeta_t}\frac{\zeta-\zeta_t}{1-{t \over (t+2)}Q(\zeta)}=
\frac{4/d}{t(2+t)}\frac{Q(\zeta_t)}{Q'(\zeta_t)}\frac{1}{\zeta_t^{1+\ell/d}}.
$$
In the limit $t \gg \max\{ 2/Q'(1), 2\}$, $\ell \gg d$, and
$\ell/t$ finite we find
\begin{eqnarray}\nonumber
f(\ell,t) \approx \frac{4}{dt^2}\frac{1}{Q'(1)}\exp \left[
-\frac{2 \ell}{d t Q'(1)} \right]&=&\frac{4}{t^2}\frac{1}{F_0'(1)}
\exp \left[ -\frac{2 \ell}{t F_0'(1)} \right] \\ \nonumber
&=&\frac{8}{\ell_0 t^2}
\exp \left( -\frac{4 \ell}{\ell_0 t} \right),
\end{eqnarray}
where we have used $F_0'(1)=Q'(1)d$, and the constant $\ell_0 = 2
F_0'(1)$ is the total number of particles $\ell_0 := M^+(0) + M^-(0)
= M^+(t) + M^-(t)$. As mentioned at the beginning of this section,
we have implicitly assumed $\ell/d \in \mathbb{N}$ because otherwise
the solution is identically zero; so the asymptotic solution
explicitly reads
$$
f(\ell,t) \approx \frac{8}{\ell_0 t^2} \exp \left( -\frac{4
\ell}{\ell_0 t} \right) \sum_{m=1}^\infty \delta_{md,\ell}.
$$
We thus see that the self-similar long-time solution ${1 \over
t^2}\Phi({\ell \over t})$ is uniform in $d$; the number $d$ is just
a measure of the size of the clusters when the self-similar state is
reached, what happens when $\ell \gg d$. Apparently, the time it
takes to reach the self-similar state should depend on the gap index
$d$: we have shown this time fulfills $t \gg \max \{ 4d/\ell_0,2\}$,
and so, the bigger the gaps are, the longer the transient to the
self-similar state would be. Note however that this condition
trivially becomes $t \gg 2$ because the gap index could be at most
twice the number of particles. In consequence the length of the
transient towards self-similarity is independent of $d$.

Now we move to calculate the self-similar asymptotics for the case
of not identically distributed initial conditions. We assume
$N(0)=1$, and recall the solution for the generating function:
\begin{eqnarray}
& & \nonumber F^{-}\left( z,t\right) \\ \nonumber &=&\frac{1}{\left(
1+\frac{t}{2}\right) ^{2}} \times \hspace{8cm} \\ & &
\frac{%
1}{\dfrac{1}{G_{0}\left( z\right) }\left[ \exp\left( -\dfrac
{G_{0}\left(
z\right) }{2}\dfrac{t}{1+\frac{t}{2}}\right) -1\right] +\dfrac{1}{%
F^{-}\left( z,0\right) }\exp\left( -\dfrac{G_{0}\left( z\right) }{2}\dfrac{t%
}{1+\frac{t}{2}}\right) }; \nonumber
\end{eqnarray}
we know
$$
f^-(\ell,t)=\frac{1}{2\pi i}
\int_{|z|=1}\frac{F^-(z,t)}{z^{\ell+1}}dz.
$$
Performing the change of variables $F^-(z,0)=P(z^d)$ and
$G_0(z)=Q(z^d)$ just like in the previous case we find
\begin{eqnarray}\nonumber
& & f^-(\ell,t) \\
&=& \frac{1}{2\pi i d}\frac{1}{\left( 1+{t \over 2}
\right)^2} \times \\ \nonumber & &
\int_{|\zeta|=1} \frac{1}{\frac{1}{Q(\zeta)}\left\{
\exp \left[ -\frac{Q(\zeta)t}{(2+t)} \right] -1
\right\}+\frac{1}{P(\zeta)}\exp \left[ -\frac{Q(\zeta)t}{(2+t)}
\right]} \frac{d\zeta}{\zeta^{1+\ell/d}}
\\ \nonumber
&=&\frac{1}{2\pi i d}\frac{1}{\left( 1+{t \over 2} \right)^2} \times
\\ \nonumber & &
\int_{|\zeta|=1+{\delta \over 2}}
\frac{1}{\frac{1}{Q(\zeta)}\left\{ \exp \left[
-\frac{Q(\zeta)t}{(2+t)} \right] -1
\right\}+\frac{1}{P(\zeta)}\exp \left[ -\frac{Q(\zeta)t}{(2+t)}
\right]} \frac{d\zeta}{\zeta^{1+\ell/d}} \\ \nonumber & &
-\frac{1/d}{\left( 1+{t \over 2} \right)^2}
\frac{1}{\zeta_t^{1+\ell/d}} \lim_{\zeta \to \zeta_t}
\frac{\zeta-\zeta_t}{\frac{1}{Q(\zeta)}\left\{ \exp \left[
-\frac{Q(\zeta)t}{(2+t)} \right] -1
\right\}+\frac{1}{P(\zeta)}\exp \left[ -\frac{Q(\zeta)t}{(2+t)}
\right]},
\end{eqnarray}
where $\delta$ is small enough and
$$
\zeta_t=1+\frac{2/t}{P'(1)+{1 \over 2}Q'(1)} \quad + \quad
\mathrm{higher} \quad \mathrm{order} \quad \mathrm{terms},
$$
is the location of the pole. Herein we have implicitly assumed
that $\ell/d \in \mathbb{N}$; otherwise $f^-(\ell,t) \equiv 0$ due
to the orthogonality of the basis of plane waves. The limit can be
calculated employing L'H\^{o}pital rule
$$
\lim_{t \to \infty} \lim_{\zeta \to \zeta_t}
\frac{\zeta-\zeta_t}{\frac{1}{Q(\zeta)}\left\{ \exp \left[
-\frac{Q(\zeta)t}{(2+t)} \right] -1
\right\}+\frac{1}{P(\zeta)}\exp \left[ -\frac{Q(\zeta)t}{(2+t)}
\right]}=-\frac{1}{P'(1)+{1 \over 2}Q'(1)}.
$$
For long times we find
$$
\frac{1}{\zeta_t^{1+\ell/d}}= \frac{1}{\left(
1+\frac{2/t}{P'(1)+{1 \over 2}Q'(1)} \right)^{1+\ell/d}} \approx
\exp \left[ -\frac{2(1+\ell/d)}{[P'(1)+{1 \over 2}Q'(1)]t}
\right].
$$
Putting all together and in the limit $\ell \gg d$, $t \gg \max
\{\frac{2}{P'(1)+Q'(1)/2},2\}=2$ (while keeping $\ell/t$ constant)
we obtain the self-similar structure
$$
\frac{4}{dt^2}\frac{1}{P'(1)+{1 \over 2}Q'(1)} \exp \left\{
-\frac{2\ell}{d[P'(1)+ {1 \over 2} Q'(1)]t} \right\}.
$$
This result is completely analogous to that of the previous case.
Taking into account that
$d[P'(1)+Q'(1)/2]=(F^-)'(1,0)+G_0'(1)/2=[(F^+)'(1,0)+(F^-)'(1,0)]/2=\ell_0/2$
is half of the total number of particles, we may write down the
general result
$$
f^\pm(\ell,t) \approx \frac{8}{t^2 \ell_0} \exp \left(
-\frac{4\ell}{\ell_0 t} \right)\sum_{m=1}^\infty \delta_{md,\ell},
$$
where the discrete Dirac comb (built as a sum of Kronecker deltas)
explicitly signals the values of $\ell$ which correspond to a
nonzero $f^\pm(\ell,t)$. We emphasize this result is identical to
the one obtained in the previous case, so this shows the asymptotic
self-similar state is independent of the initial distributions of
clusters $f^\pm(\ell,0)$ except for their zeroth mode (i.~e. the
number of clusters). Also as in the previous case, the cluster sizes
in the self-similar regime fulfill $\ell \gg d$, and the transient
time fulfills $t \gg 2$. Note that the total number of particles is
as always conserved.

\subsection{Convergence}\label{convergence}

In this section we prove rigorous convergence results to the self-similar profile formally calculated in the previous section
$$
\Phi \left( {\ell \over t} \right)=\frac{8}{\ell_0} \exp \left(
-\frac{4\ell}{\ell_0 t} \right)\sum_{m=1}^\infty \delta_{md,\ell}.
$$
An analogous convergence result for the case examined in section~\ref{formalsymm} is an immediate corollary of the present result,
so we will not explicitely consider it here.

\begin{theorem}\label{thefun1}
There exists a suitable constant $C''$ such that
$$
\sum_{\ell=1}^{\infty} \ell^p \left| f^{\pm}(\ell,t)-
\frac{1}{t^2}\Phi \left( {\ell \over t} \right) \right| \le
\frac{C''}{t^2}.
$$
for all $1 \le p < \infty$ in the long time limit $t \to
\infty$, provided all moments
of the initial condition are bounded.
\end{theorem}

\begin{proof}
We found in the previous section
\begin{eqnarray}\nonumber
& & f^-(\ell,t) \\ \nonumber &=& \frac{1}{2\pi i d}\frac{1}{\left( 1+{t \over 2}
\right)^2} \times \\ \nonumber &&
\int_{|\zeta|=1}
\frac{1}{\frac{1}{Q(\zeta)}\left\{ \exp \left[
-\frac{Q(\zeta)t}{(2+t)} \right] -1
\right\}+\frac{1}{P(\zeta)}\exp \left[ -\frac{Q(\zeta)t}{(2+t)}
\right]} \frac{d\zeta}{\zeta^{1+\ell/d}} \\ \nonumber & &
+\frac{1/d}{\left( 1+{t \over 2} \right)^2}
\frac{1}{\zeta_t^{1+\ell/d}} \lim_{\zeta \to \zeta_t}
\frac{\zeta-\zeta_t}{\frac{1}{Q(\zeta)}\left\{ \exp \left[
-\frac{Q(\zeta)t}{(2+t)} \right] -1
\right\}+\frac{1}{P(\zeta)}\exp \left[ -\frac{Q(\zeta)t}{(2+t)}
\right]}
\\ \nonumber
&=&\frac{1}{2\pi i d}\frac{1}{\left( 1+{t \over 2} \right)^2} \times
\\ \nonumber
& & \int_{|\zeta|=1+{\delta \over 2}}
\frac{1}{\frac{1}{Q(\zeta)}\left\{ \exp \left[
-\frac{Q(\zeta)t}{(2+t)} \right] -1
\right\}+\frac{1}{P(\zeta)}\exp \left[ -\frac{Q(\zeta)t}{(2+t)}
\right]} \frac{d\zeta}{\zeta^{1+\ell/d}}.
\end{eqnarray}
This equality holds as a consequence of the finiteness we imposed
on the different moments of the initial condition, i.~e.
$$
\sum_{\ell=1}^\infty \ell^p f^\pm(\ell,0) < \infty \qquad 1 \le p <
\infty.
$$
Note this is equivalent to requiring the following bound for the
initial condition
$$
f^{\pm}(\ell,0) \le \frac{C}{(1+\delta)^{\ell}}
$$
for any $C,\delta>0$, see \cite{penrose}. This assures that the
transformed functions $F^{\pm}(z,t)$ are analytic in the open disc
$|z|<1+\delta$, and so the contour deformation of the above
complex integrals makes sense (note that this requirement is
sharp). If there are no more poles in the integrand for
$|z|<1+\delta$ apart from the aforementioned roots of unity,
$z^{d}=1$, we have
$$
\left| f^{\pm}(\ell,t)- \frac{1}{t^2}\Phi \left( {\ell \over t}
\right) \right| \le \frac{C'}{t^2}\left(1 + {\delta \over 2}
\right)^{-\ell/d},
$$
for some constant $C'$ and where
$$
\frac{1}{t^2}\Phi \left( {\ell \over t} \right)=\frac{8}{t^2 \ell_0}
\exp \left( -\frac{4\ell}{\ell_0 t} \right)\sum_{m=1}^\infty \delta_{md,\ell}
$$
is the self-similar form. This result implies pointwise
convergence of the solution $f^{\pm}(\ell,t)$ to the self-similar
form $t^{-2}\Phi(\ell/t)$ uniformly in $\ell$. Furthermore, the
decay in $\ell$ is strong enough to have the estimate stated in
the theorem.
\end{proof}

\begin{remark}
Note that the integral
$$
\int_{|\zeta|=1+{\delta \over 2}}
\frac{1}{\frac{1}{Q(\zeta)}\left\{ \exp \left[
-\frac{Q(\zeta)t}{(2+t)} \right] -1
\right\}+\frac{1}{P(\zeta)}\exp \left[ -\frac{Q(\zeta)t}{(2+t)}
\right]} \frac{d\zeta}{\zeta^{1+\ell/d}}
$$
can be bounded by a constant $C$ if the integrand contains no
poles. This is not necessarily the case for $|\zeta|=1+\delta/2$.
But taking into account that $Q(\zeta)$ and $P(\zeta)$ are
holomorphic in $|z|<1+\delta$ implies that the denominator (which
is not zero on $|\zeta| \le 1, \zeta \neq 1$) is meromorphic, and
thus it has a finite number of zeros in the open disk
$|\zeta|<1+\delta$. A direct consequence of this fact is that we
can choose some $0<\delta_1<\delta$ small enough such that the
integral along the line $|\zeta|=1+\delta_1$ has no poles and it
encloses an area in which $\zeta=1$ is the only pole.
\end{remark}

\subsection{The Sharp Case}\label{sharp}

In this section we are interested in obtaining asymptotic
convergence results with the weakest possible hypothesis on $f^\pm
\left( \ell ,0\right)$. In particular, we will assume the
finiteness of the first moment of the initial condition only:
\begin{equation*}
\sum_{\ell =1}^{\infty }\ell f^\pm \left( \ell ,0\right) <\infty.
\end{equation*}
Note this is equivalent to have initial conditions of the form
\begin{equation*}
f^\pm \left( \ell ,0\right) \approx \frac{1}{\ell ^{\alpha }},
\end{equation*}
except for slow variation functions, for $1<\alpha \leq 2$. No
further assumptions will be imposed in this section.
As in the previous section we will only rigorously prove the formal asymptotics in
section~\ref{formalasymm}, as the corresponding result for section~\ref{formalsymm} is
a direct corolary of this proof.

We start proving the following technical result:

\begin{lemma}
$Q(z^d) \in C^1 \left( \overline{B_1(0)} \right)$.
\end{lemma}

\begin{proof}
Let us remind that
$$
|Q'(z^d)|=|(F^+)'(z,0)-(F^-)'(z,0)| \le \sum_{\ell=1}^\infty
f^{+}(\ell,0) \ell + \sum_{\ell=1}^\infty f^{-}(\ell,0) \ell <
\infty,
$$
for all $z \in \overline{B_1(0)}$. Using this fact together with
the analyticity of $Q(z^d)$ in $B_1(0)$ the result follows.
\end{proof}

\begin{corollary}
As a consequence we may write
$$
Q(\xi)=Q'(1)(\xi-1)+r(\xi),
$$
where the rest function $r(\xi)=o \left( |\xi-1| \right)$ as $\xi
\to 1$, uniformly in $\xi$ for $|\xi| \le 1$.
\end{corollary}

\begin{remark}
Note we equivalently have
$$
P(\xi)=1+P'(1)(\xi-1)+R(\xi),
$$
where the rest function $R(\xi)=o \left( |\xi-1| \right)$ as $\xi
\to 1$, uniformly in $\xi$ for $|\xi| \le 1$.
\end{remark}

Now we list the main result of this section

\begin{theorem}\label{thefun2}
The following convergence property
$$
\sup_{\ell \ge 1} \, \ell \left| f^{\pm}(\ell,t)-
\frac{1}{t^2}\Phi \left( {\ell \over t} \right) \right| =
o\left(t^{-1}\right)
$$
holds true in the long time limit $t \to \infty$, provided the first moment of the
initial condition is bounded.
\end{theorem}

\begin{proof}
Consider our solution
$$
f^-(\ell,t)=\frac{1}{2 \pi i d}
\frac{1}{(1+t/2)^2}\int_{|\zeta|=1}\frac{1}{H(\zeta,t)}\frac{d\zeta}{\zeta^{1+\ell/d}},
$$
where
$$
H(\zeta,t)=\frac{1}{Q(\zeta)}\left\{ \exp \left[
-\frac{Q(\zeta)t}{(2+t)} \right] -1
\right\}+\frac{1}{P(\zeta)}\exp \left[ -\frac{Q(\zeta)t}{(2+t)}
\right].
$$
Expanding around $\zeta = 1$ we find
$$
H(\zeta,t)=\frac{2}{2+t}-(\zeta-1)\lambda(t)-R(\zeta)-a(t)r(\zeta)+O(|\zeta-1|^2),
$$
and
$$
\frac{1}{H(\zeta,t)}-\frac{1}{\frac{2}{2+t}-\lambda(t)(\zeta-1)}=\frac{a(t)r(t)-S(\zeta,t)}{H(\zeta,t)\left[\frac{2}{2+t}-\lambda(t)(\zeta-1)
\right]},
$$
where
\begin{eqnarray} \nonumber
S(\zeta,t) &=& H(\zeta,t)+\lambda(t)(\zeta-1)+a(t)r(\zeta)-\frac{2}{2+t}=-R(\zeta)+O(|\zeta-1|^2),
\\ \nonumber
\lambda(t) &=& P'(1)+Q'(1) \frac{t(4+t)}{2(2+t)^2},
\end{eqnarray}
and
$$
a(t) = \frac{t(4+t)}{2(2+t)^2}.
$$
Substituting we find
\begin{eqnarray}\nonumber
f^-(\ell,t)&=&\frac{1}{2 \pi i d}
\frac{1}{(1+t/2)^2}\int_{|\zeta|=1}
\frac{1}{\frac{2}{2+t}-\lambda(t)(\zeta-1)}\frac{d
\zeta}{\zeta^{1+\ell/d}} \\ \nonumber & &
-\frac{1}{2 \pi i
d}\frac{1}{(1+t/2)^2}\int_{|\zeta|=1}
\frac{a(t)r(\zeta)-S(\zeta,t)}{H(\zeta,t)\left[
\frac{2}{2+t}-\lambda(t)(\zeta-1)
\right]}\frac{d\zeta}{\zeta^{1+\ell/d}}.
\end{eqnarray}
The first summand yields the self-similar profile
$$
\frac{1}{2 \pi i d} \frac{1}{(1+t/2)^2}\int_{|\zeta|=1}
\frac{1}{\frac{2}{2+t}-\lambda(t)(\zeta-1)}\frac{d
\zeta}{\zeta^{1+\ell/d}}=\frac{1}{t^2}\Phi \left( \frac{\ell}{t}
\right),
$$
and the second summand the rest
$$
g_\ell(t)=\frac{1}{2 \pi i d}\frac{1}{(1+t/2)^2}\int_{|\zeta|=1}
\frac{-a(t)r(\zeta)+S(\zeta,t)}{H(\zeta,t)\left[
\frac{2}{2+t}-\lambda(t)(\zeta-1)
\right]}\frac{d\zeta}{\zeta^{1+\ell/d}}.
$$
Let us first derive the self-similar profile
\begin{eqnarray} \nonumber
& & \frac{1}{2 \pi i d} \frac{1}{(1+t/2)^2}\int_{|\zeta|=1}
\frac{1}{\frac{2}{2+t}-\lambda(t)(\zeta-1)}\frac{d
\zeta}{\zeta^{1+\ell/d}} \\ \nonumber
&=&\frac{1}{2 \pi i d}
\frac{1}{(1+t/2)^2}\int_{|\zeta|=1+\delta}
\frac{1}{\frac{2}{2+t}-\lambda(t)(\zeta-1)}\frac{d
\zeta}{\zeta^{1+\ell/d}} \\ \nonumber
& & -\frac{1/d}{(1+t/2)^2}\mathrm{Res}
\left[
\frac{1}{\frac{2}{2+t}-\lambda(t)(\zeta-1)}\frac{1}{\zeta^{1+\ell/d}}
,\zeta=\zeta_t \right],
\end{eqnarray}
where we have implicitly assumed that $\ell/d$ is an integer
(otherwise the integral vanishes) and where
$$
\zeta_t = 1+\frac{4/t}{2P'(1)+Q'(1)}, \qquad \mathrm{when} \qquad
t \to \infty,
$$
and $\delta>0$ is large enough so that the pole at $\zeta_t$ lies in
the area enclosed by $|z|=1+\delta$. Computing the limit at the
residue by means of l'H\^{o}pital rule we find
$$
-\frac{1/d}{(1+t/2)^2}\mathrm{Res} \left[
\frac{1}{\frac{2}{2+t}-\lambda(t)(\zeta-1)}\frac{1}{\zeta^{1+\ell/d}}
,\zeta=\zeta_t \right]=\frac{8}{t^2 \ell_0} \exp \left(
-\frac{4\ell}{\ell_0 t} \right),
$$
as in the previous cases, and the integral may be bounded as in the
last section as well
$$
\frac{1}{2 \pi i d} \frac{1}{(1+t/2)^2}\int_{|\zeta|=1+\delta}
\frac{1}{\frac{2}{2+t}-\lambda(t)(\zeta-1)}\frac{d
\zeta}{\zeta^{1+\ell/d}} \le \frac{C}{t^2}(1+\delta)^{-\ell/d},$$
leading us to the result
\begin{eqnarray}\nonumber
& & \frac{1}{2 \pi i d} \frac{1}{(1+t/2)^2}\int_{|\zeta|=1}
\frac{1}{\frac{2}{2+t}-\lambda(t)(\zeta-1)}\frac{d
\zeta}{\zeta^{1+\ell/d}} \\ \nonumber
&=& \frac{8}{t^2 \ell_0} \exp \left( -\frac{4\ell}{\ell_0 t} \right)
\sum_{m=1}^\infty \delta_{md,\ell} + \tilde{g}_\ell(t), \qquad
\tilde{g}_\ell(t) \le \frac{C}{t^2}(1+\delta)^{-\ell/d},
\end{eqnarray}
for long times and uniformly in $\ell$. Note that we have made
explicit the assumption of $\ell/d$ being an integer by means of
the introduction of the Dirac comb. Let us now return to the rest
$g_\ell(t)$. We have the following bounds on it (the constant $C$
may change from line to line):
$$
\left| \ell g_\ell(t) \right| \le \frac{C}{t^2} \left|
\int_{|\zeta|=1} \frac{S_1(\zeta,t)}{H(\zeta,t)\left[
\frac{2}{2+t}-\lambda(t)(\zeta-1) \right]} \frac{\ell
d\zeta}{\zeta^{1+\ell/d}} \right|,
$$
where $S_1(\zeta,t)=-a(t)r(\zeta)+S(\zeta,t)$. Integrating by
parts we have
\begin{eqnarray}\nonumber
& & \frac{C}{t^2} \left| \int_{|\zeta|=1}
\frac{S_1(\zeta,t)}{H(\zeta,t)\left[
\frac{2}{2+t}-\lambda(t)(\zeta-1) \right]} \frac{\ell
d\zeta}{\zeta^{1+\ell/d}} \right| \\ \nonumber
&\le& \frac{C}{t^2} \left|
\int_{|\zeta|=1} \frac{\partial}{\partial \zeta} \left\{
\frac{S_1(\zeta,t)}{H(\zeta,t)\left[
\frac{2}{2+t}-\lambda(t)(\zeta-1) \right]} \right\}
\frac{d\zeta}{\zeta^{\ell/d}} \right|
\\ \nonumber &\le&
\frac{C}{t^2} \int_{|\zeta|=1} \frac{ |\partial_\zeta
S_1(\zeta,t)|}{|H(\zeta,t)|\left|
\frac{2}{2+t}-\lambda(t)(\zeta-1) \right|} |d\zeta| +
\\ \nonumber
& & \frac{C}{t^2} \int_{|\zeta|=1} |S_1(\zeta,t)| \frac{\left|
H_\zeta(\zeta,t)\left[
\frac{2}{2+t}-\lambda(t)(\zeta-1)\right]-H(\zeta,t)\lambda(t)
\right|}{H(\zeta,t)^2\left[ \frac{2}{2+t}-\lambda(t)(\zeta-1)
\right]^2} |d\zeta|.
\end{eqnarray}
To bound these integrals we take into account
$$
|H(\zeta,t)| \ge \epsilon_0 \left( |\zeta-1| +t^{-1} \right),
\qquad \left| \frac{2}{2+t}-\lambda(t)(\zeta-1) \right| \ge
\epsilon_0 \left( |\zeta-1| +t^{-1} \right),
$$
where $\epsilon_0 > 0$ is a small enough constant and this last
quantity, together with $\lambda(t)$, admits also a constant upper
bound. These considerations lead to
\begin{equation}
\left| \ell g_\ell(t) \right| \le \frac{C}{t^2}
\int_{|\zeta|=1}\frac{|\partial_\zeta
S_1(\zeta,t)|}{(|\zeta-1|+t^{-1})^2}|d\zeta|+\frac{C}{t^2}
\int_{|\zeta|=1}\frac{|
S_1(\zeta,t)|}{(|\zeta-1|+t^{-1})^3}|d\zeta|.
\end{equation}
Using that $S_1(\zeta,t)=o(|\zeta-1|)$ and $\partial_\zeta
S_1(\zeta,t)=o(1)$ we find
\begin{eqnarray}\nonumber
& & \frac{C}{t^2} \int_{|\zeta-1| \le \delta_1}\frac{|\partial_\zeta
S_1(\zeta,t)|}{(|\zeta-1|+t^{-1})^2}|d\zeta|+\frac{C}{t^2}
\int_{|\zeta-1| \le \delta_1}\frac{|
S_1(\zeta,t)|}{(|\zeta-1|+t^{-1})^3}|d\zeta|
\\ \nonumber &\le&
\frac{o(1)}{t^2} \int_{|\zeta-1| \le
\delta_1}\frac{1}{(|\zeta-1|+t^{-1})^2}|d\zeta|+\frac{o(1)}{t^2}
\int_{|\zeta-1| \le \delta_1}
\frac{|\zeta-1|}{(|\zeta-1|+t^{-1})^3}|d\zeta| \\ \nonumber
&\le& \frac{o(1)}{t},
\end{eqnarray}
where $\delta_1 > 0$ is small enough. On the other hand we have
$$
\frac{C}{t^2} \int_{|\zeta-1| \ge \delta_1}\frac{|\partial_\zeta
S_1(\zeta,t)|}{(|\zeta-1|+t^{-1})^2}|d\zeta|+\frac{C}{t^2}
\int_{|\zeta-1| \ge \delta_1}\frac{|
S_1(\zeta,t)|}{(|\zeta-1|+t^{-1})^3}|d\zeta| \le \frac{C}{t^2},
$$
because $\partial_\zeta S_1(\zeta,t)$ and $S_1(\zeta,t)$ are
bounded on $\{|\zeta|=1\} \cap \{\zeta \neq 1\}$, as found in
section~\ref{asymm}, where we have shown that $H(\zeta,t)$
is free from zeros in $\{|\zeta| \le 1\} \cap \{\zeta \neq 1\}$
uniformly in $t$. As a consequence we have
$$
\sup_{\ell \ge 1} \left| \ell \, g_\ell(t) \right| = o \left(
\frac{1}{t} \right).
$$
As these estimates hold identically for $f^+(\ell,t)$ the final
result follows.
\end{proof}

\newpage

\section{Self-similar asymptotics for the majority kernel}
\label{majoritykernel}

\subsection{Formal calculations concerning self-similarity of the solution}

In this section we concentrate in the model given by Eq.~(\ref
{collisionkernel3}). Let us remind that the key result in this
case was the conservation of the number of particles travelling in
either direction given by Eq.~(\ref{noorder}). This phenomenology
is reminiscent of the one given by the random kernel for the
symmetric initial condition, for which order is not found at the
kinetic level either. So, given the last section results, one
would expect self-similarity of the solution in the case of the
majority kernel too, but now independently of the initial
condition.

Following this reasoning, and given the scaling found for the
solutions corresponding to the random kernel in the last section,
we formally propose the following scaling behavior
\begin{equation*}
F^+(z,t)=\frac{1}{t} \varphi^+[(z-1)t], \qquad F^-(z,t)=\frac{1}{t}
\varphi^-[(z-1)t],
\end{equation*}
for the solutions $F^\pm(z,t)$ to
Eqs.~(\ref{majority1})-(\ref{majority2}). The scaling functions
would obey the system

\begin{eqnarray}
\xi \partial_{\xi \xi} \varphi^+= [\varphi^- -\varphi^-(0)]
\partial_{\xi}
\varphi^+, \\
\xi \partial_{\xi \xi} \varphi^-=[\varphi^+ -\varphi^+(0)]
\partial_{\xi} \varphi^-,
\end{eqnarray}
where $\xi=(z-1)t$ is the self-similar variable. Note that making the substitution $W^+=\varphi^+ -\varphi^+(0)$ and $%
W^-=\varphi^- -\varphi^-(0)$ we arrive at the following system of
ordinary differential equations (ODEs)

\begin{equation*}
\xi \partial_{\xi \xi} W^+ = W^- \partial_\xi W^+, \qquad \xi \partial_{\xi \xi} W^- = W^+
\partial_\xi W^-.
\end{equation*}
Changing variables $\xi=-e^\tau$ we arrive at the four-dimensional
differential system

\begin{eqnarray}
\partial_{\tau} H^+ &=& W^- H^+ + H^+, \\
\partial_{\tau} W^+ &=& H^+, \\
\partial_{\tau} H^- &=& W^+ H^- + H^-, \\
\partial_{\tau} W^- &=& H^-.
\end{eqnarray}
By changing variables again $\psi^\pm=W^\pm+1$ we find the new
differential system
\begin{eqnarray}
\partial_{\tau} H^+ &=& \psi^- H^+, \\
\partial_{\tau} \psi^+ &=& H^+, \\
\partial_{\tau} H^- &=& \psi^+ H^-, \\
\partial_{\tau} \psi^- &=& H^-,
\end{eqnarray}
subject to the initial conditions
$$
\psi^\pm(-\infty)=1, \qquad H^\pm(- \infty)=0,
$$
and which long time behavior is
$$
H^\pm(+ \infty)=0, \qquad \psi^\pm(+\infty)=1-\varphi^\pm(0).
$$
We will interpret this differential problem as a dynamical system.
In this dynamical system one finds two invariant hyperplanes
$\{H^\pm = 0\}$ and one invariant plane $\{\psi^+=\psi^-, \,
H^+=H^-\}$. All the fixed points of this system sit in the plane $\{H^+=H^-=0\}$;
in fact, this is a degenerated plane with all its points being fixed for
the dynamical system. This system admits the first
integral of motion $E=\psi^+\psi^- -H^+ -H^-$, and it takes the value $E=1$
for the initial conditions we are considering.
All this implies that the orbits we are interested in,
that are attracted by the plane $\{H^+=H^-=0\}$
in both the positive and negative infinite time limits, approach
the hyperbola $\psi^+ \psi^-=1$ when $\tau \to \pm \infty$.

In the case $\varphi^\pm(0)=2$ we can find an analytic expression
for the heteroclinic connection. In this case $\psi^+=\psi^-
\equiv \psi$ and $H^+=H^- \equiv H$, and so
$$
\partial_{\tau} \psi= \frac{\psi^2-1}{2}.
$$
We find the solution
$$
\psi(\tau)=-\tanh \left( \frac{\tau-\tau_0}{2} \right), \qquad
H(\tau)=\frac{-1}{1+\cosh(\tau-\tau_0)},
$$
where $\tau_0$ is an arbitrary real constant. In the original
variables the solution reads
$$
F^+(z,t)=F^-(z,t)=\frac{1}{t}\frac{2 \xi_0}{(z-1)t+\xi_0},
$$
where $\xi_0=-e^{\tau_0}$ is an arbitrary negative constant.

In the general case we have not been able to obtain the profiles
$\psi^\pm$ analytically, but we can anyway recover the original
variables to find
$$
F^\pm(z,t)=\frac{1}{t} \left\{ \varphi^\pm(0)-1+\psi^\pm \left[
\ln \left( \frac{(z-1)t}{\xi_0} \right) \right] \right\},
$$
or alternatively
$$
F^\pm(z,t)=\frac{1}{t} \left\{\psi^\pm \left[ \ln \left(
\frac{(z-1)t}{\xi_0} \right) \right] -\psi^\pm(+\infty) \right\},
$$
where $\xi_0$ is an arbitrary negative constant. In the following
section we will examine in a more rigorous way the existence of
these formal expressions as well as their significance.

\subsection{Continuum Dynamics and Laplace Transform}

First of all we note that the self-similar form we have calculated
in the last section corresponds to the solution of the continuum
version of Eq.~(\ref{collisionkernel3}), which reads
\begin{equation}\label{continuum}
\partial_t f^\pm (x,t)= \int_0^x \frac{y}{x}
f^\pm(y,t)f^\mp (x-y,t)dy -f^\pm(x,t) \int_0^\infty f^\mp (y,t)dy.
\end{equation}
In this section we will focus on this continuum version rather
than on the discrete dynamics. We note that both dynamics,
continuum and discrete, should behave analogously in the case of
clusters composed by large numbers of particles, and in particular
this number must much larger than $d$.

The Laplace transform of the solution to Eq.~(\ref{continuum}) is
$$
\hat{f}^\pm(z,t)=\int_0^\infty f^\pm(x,t) e^{-zx} dx,
$$
which is defined for $\mathrm{Re}(z)>0$. The Laplace transformed
version of this equation reads
$$
\partial_t \partial_z \hat{f}^\pm(z,t)=\hat{f}^\mp(z,t)
\partial_z \hat{f}^\pm(z,t) - \hat{f}^\mp(0,t) \partial_z
\hat{f}^\pm(z,t).
$$
According to the formal developments in the previous section, we assume
the following self-similar form of the solution in Laplace space
$$
\hat{f}^\pm(z,t)=\frac{1}{t} \varphi^\pm(zt)=\frac{1}{t}
\varphi^\pm(\xi),
$$
where $\xi=zt$ is the self-similar variable. This corresponds to the
following self-similar form in real space
\begin{equation}\label{scaling}
f^\pm(x,t)=\frac{1}{t^2} \Phi^\pm \left( \frac{x}{t} \right) \, +
\, \text{correction terms,}
\end{equation}
that is reminiscent of the one found in the analysis of the random
kernel. Note the self-similar profiles $\Phi^\pm$ can be recovered
from the inverse Laplace transform
\begin{equation}\label{bromwich}
\Phi^\pm(\zeta)=\frac{1}{2 \pi i} \int_{\gamma -i\infty}^{\gamma
+i\infty} e^{\zeta \eta} \varphi^\pm(\eta) d\eta,
\end{equation}
where $\gamma \in \mathbb{R}$ is large enough and we have used the
Bromwich integral formula~\cite{arfken}.

Repeating the calculations of the previous section, employing the
same notation, we reduce the problem to studying the differential
system
\begin{subequations}
\begin{eqnarray}
\label{system1}
\partial_{\tau} H^+ &=& \psi^- H^+, \\
\label{system2}
\partial_{\tau} \psi^+ &=& H^+, \\
\label{system3}
\partial_{\tau} H^- &=& \psi^+ H^-, \\
\label{system4}
\partial_{\tau} \psi^- &=& H^-,
\end{eqnarray}
\end{subequations}
subject to the conditions
$$
\psi^\pm(-\infty)=1, \qquad H^\pm(- \infty)=0,
$$
and
$$
H^\pm(+ \infty)=0, \qquad \psi^\pm(+\infty)=1-\varphi^\pm(0).
$$
As we have shown in the previous section this system encodes the
form of the self-similar profiles $\Phi^\pm$. This differential
system, as a problem concerning real functions of real variables,
was already studied in the previous section.

In order to proceed with the inverse Laplace transform we need to
extend the solution to this differential system to the complex
plane. The idea is to perform an analytic continuation by
considering the solution of Eqs.~(\ref{system1}, \ref{system2},
\ref{system3}, \ref{system4}), whose trajectories numerically define
the self-similar profiles $\varphi^\pm$, to $\mathbb{C}$. In order
to do so one needs to be sure that this system is free of
singularities in a suitable region of the complex plane. In
particular, the inverse Laplace transform formula~(\ref{bromwich})
makes sense for those solutions free of singularities for
$\mathrm{Re}(\xi)>0$. We have checked this is actually true by numerically integrating
the differential system. The solution for
the ``$+$'' fields for a particular initial condition is represented
in Figs. \ref{p1r}, \ref{p1i}, \ref{h1r} and \ref{h1i}; the solution
of the ``$-$'' fields is completely analogous to this one. The
initial condition was chosen to be the fixed point which serves as
the $\alpha-$limit of the heteroclinic connection in real variables
plus a small positive perturbation in both $\mathrm{Im}(H^+)$ and
$\mathrm{Im}(H^-)$. If both perturbations are of negative sign then
the solution behaves analogously, but if they are of opposite sign
then the solution grows unboundedly. By using a family of initial
conditions, we have numerically checked that in fact there are no
singularities in the whole strip $|\mathrm{Im} (\tau)| \le \pi/2$,
$\mathrm{Re}(\tau)\in(-\infty, \infty)$. This implies that the
functions $\varphi^\pm$, considered as functions of $\xi$
(considered also as a complex variable) are analytic in the
half-plane $\mathrm{Re}(\xi)>0$ and therefore the function
$\Phi^\pm$ can be obtained by means of the inversion formula for the
Laplace transform, Eq.~(\ref{bromwich}). This numerically shows the
existence of scaling solutions to coagulation Eq.~(\ref{continuum})
obeying the self-similar scaling~(\ref{scaling}). On the other hand
we note have clarified under which conditions these solutions are
selected. It is reasonable to expect that for sufficiently symmetric
initial conditions they will indeed be selected. However it is not
so clear that the same will happen for very asymmetric initial
conditions. We leave this question as an open problem.

\begin{figure}[tbp]
\begin{center}
\psfig{file=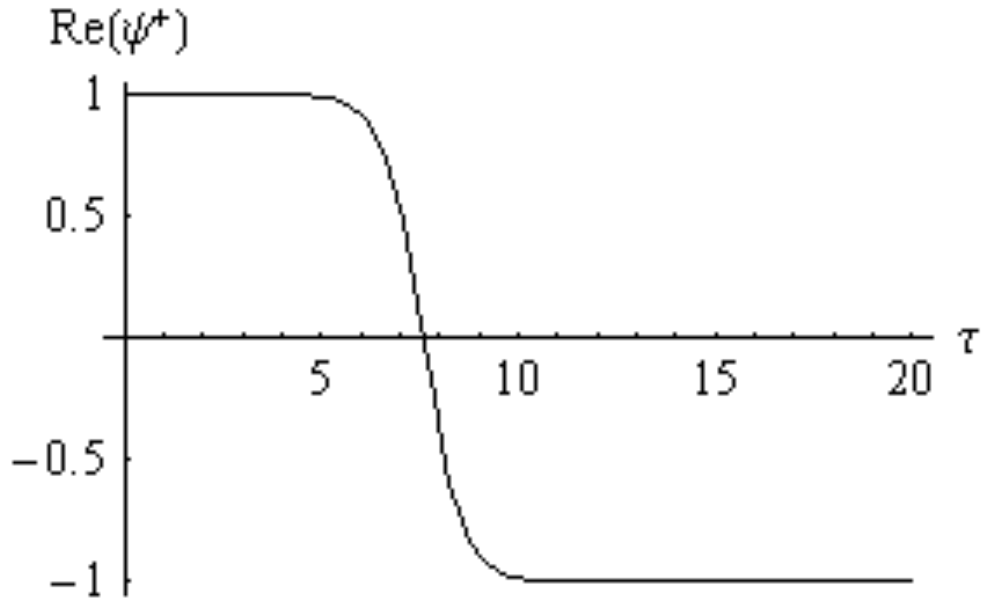,width=8cm,angle=0}
\end{center}
\caption{Numerical solution $\mathrm{Re}(\psi^+)$ versus $\tau$ of
system Eqs.~(\ref{system1}, \ref{system2},
\ref{system3}, \ref{system4}) integrated in the complex plane. The initial
conditions are $H^+(0)=10^{-3}\,i$, $H^-(0)=10^{-3}\,i$, $\psi^+(0)=1$, $\psi^-(0)=1$.} \label{p1r}
\end{figure}

\begin{figure}[tbp]
\begin{center}
\psfig{file=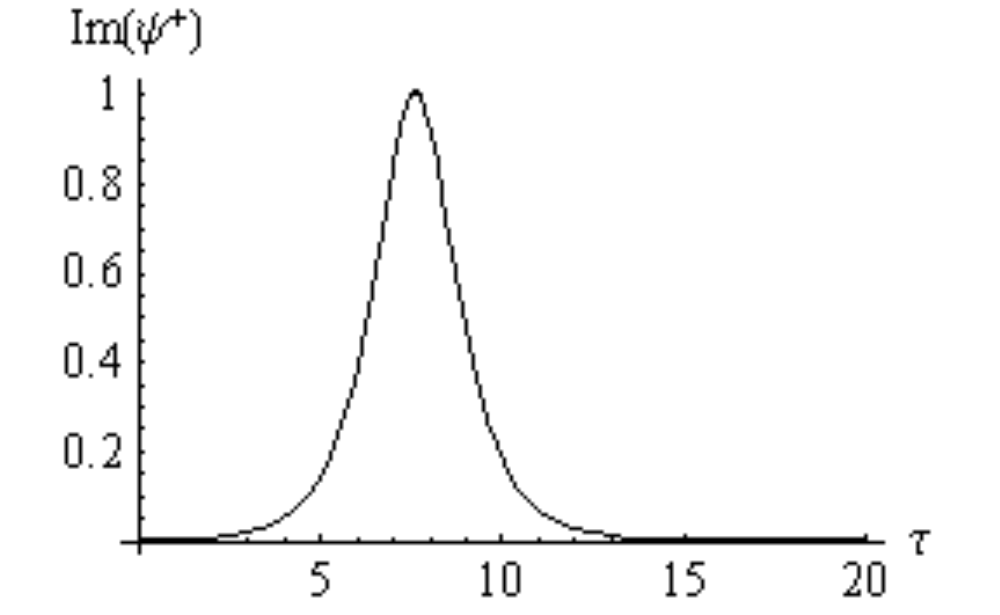,width=8cm,angle=0}
\end{center}
\caption{Numerical solution $\mathrm{Im}(\psi^+)$ versus $\tau$ of
system Eqs.~(\ref{system1}, \ref{system2},
\ref{system3}, \ref{system4}) integrated in the complex plane. The initial
conditions are $H^+(0)=10^{-3}\,i$, $H^-(0)=10^{-3}\,i$, $\psi^+(0)=1$, $\psi^-(0)=1$.} \label{p1i}
\end{figure}

\begin{figure}[tbp]
\begin{center}
\psfig{file=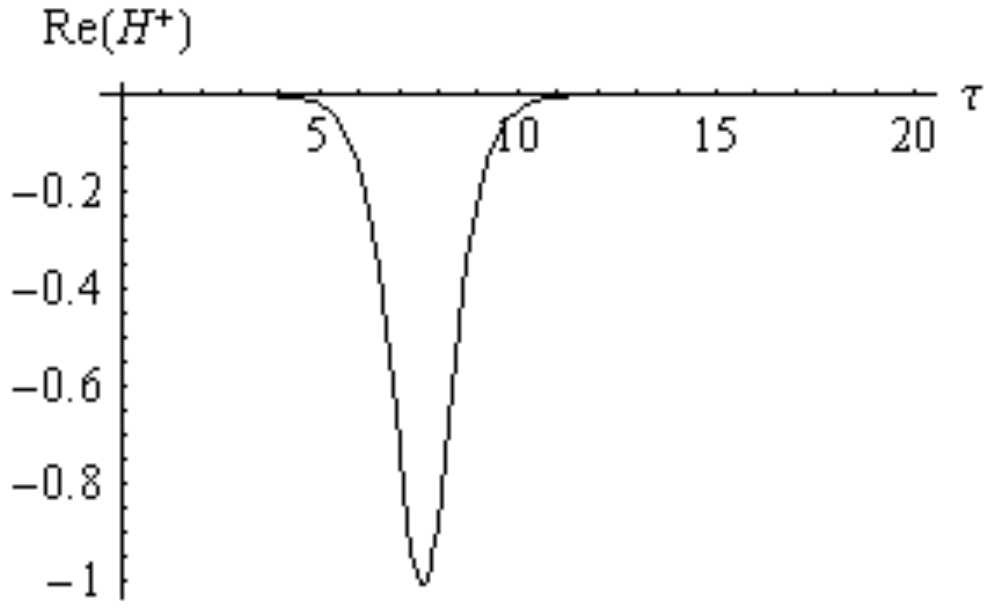,width=8cm,angle=0}
\end{center}
\caption{Numerical solution $\mathrm{Re}(H^+)$ versus $\tau$ of
system Eqs.~(\ref{system1}, \ref{system2},
\ref{system3}, \ref{system4}) integrated in the complex plane. The initial
conditions are $H^+(0)=10^{-3}\,i$, $H^-(0)=10^{-3}\,i$, $\psi^+(0)=1$, $\psi^-(0)=1$.} \label{h1r}
\end{figure}

\begin{figure}[tbp]
\begin{center}
\psfig{file=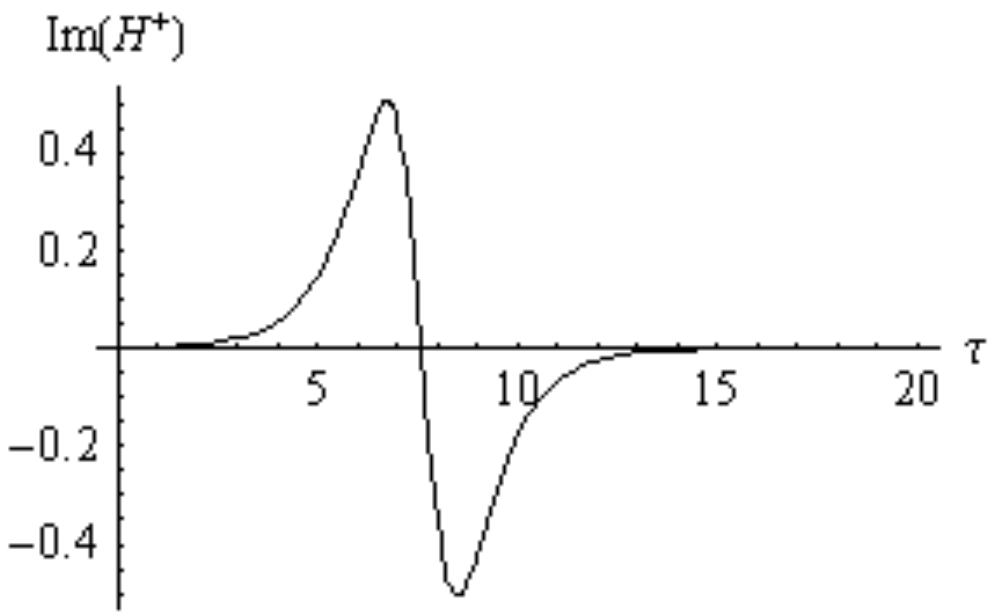,width=8cm,angle=0}
\end{center}
\caption{Numerical solution $\mathrm{Im}(H^+)$ versus $\tau$ of
system Eqs.~(\ref{system1}, \ref{system2},
\ref{system3}, \ref{system4}) integrated in the complex plane. The initial
conditions are $H^+(0)=10^{-3}\,i$, $H^-(0)=10^{-3}\,i$, $\psi^+(0)=1$, $\psi^-(0)=1$.} \label{h1i}
\end{figure}

As a final remark let us mention that it is possible to
analytically show that the functions $\varphi^\pm$ (considered as
functions of the variable $\tau$, which is now considered to be
complex) can be analytically extended to the region
$|\mathrm{Re}(\tau)|>L$, $|\mathrm{Im}(\tau)| \le \pi/2$ for some
real $L$ sufficiently large. This is so because near the fixed
points one can analytically continue the system solution by means
of a linear stability analysis. In this case we have
$$
\psi^\pm(\tau) \approx \psi^\pm(+\infty \,\, \mathrm{or} \,
-\infty)+ \phi^\pm(\tau),
$$
where the $\phi^\pm(\tau)$ are small functions. In this case we
can solve the system to find
$$
H^\pm(\tau)=H^\pm(\tau_1) \exp \left[
(\tau-\tau_1)\psi^\mp(+\infty \,\, \mathrm{or} \, -\infty) +
\int_{\tau_1}^\tau \phi^\mp(s)ds \right],
$$
$$
\phi^\pm (\tau)=\phi^\pm (\tau_1)+\int_{\tau_1}^\tau H^\pm(s)ds.
$$
We can be sure that the solution to the linearized system is free
of singularities, and thus it can be used to safely analytically
continue the solution into a neighborhood of $\tau= +\infty \,\,
\mathrm{and} \, -\infty$ in the complex plane.

\section{Numerical results}
\label{numericalresults}

\subsection{Details of the simulation}

In order to test the applicability of some of the theoretical results, we have performed extensive numerical simulations of a particular
aggregation process. In our model, we consider initially $N(t=0)\equiv N_0$ clusters, each one with only one particle. They are randomly and uniformly placed on the interval $[0,1]$ (we use periodic boundary conditions). Once the cluster locations have been set, we assign to them velocities $+1$ or $-1$. We denote by $p_0$ the initial fraction of clusters that have velocity $+1$ and, consequently, $1-p_0$ is the initial fraction with velocity $-1$. Given $p_0$ we have considered two different ways in the setting
of the initial condition:

\bigskip

(i) We run over the initial $N_0$ clusters and, for each one of them, we draw a random number $u$ extracted from a uniform distribution in the $(0,1)$ interval. If $u<p_0$ we assign the cluster the velocity $+1$, otherwise it is given the velocity $-1$. In this case the number of clusters $ N_+(t=0)$ that initially have the velocity $+1$ is a random variable that follows a binomial distribution with average value $\langle N_+(t=0)\rangle=p_0N_0$ and variance
$\sigma^2[N_+(t=0)]=p_0(1-p_0)N_0$ . Similarly, the number of clusters $N_-(t=0)$ which initially have velocity $-1$ follows a binomial distribution with average value $\langle N_-(t=0)\rangle=(1-p_0)N_0$ and variance $\sigma^2[N_-(t=0)]=p_0(1-p_0)N_0$. If we define $Z(t)=N_+(t)-N_-(t)$, we note that $Z(t=0)$ also follows a
binomial distribution with mean value $\langle
Z(t=0)\rangle=(2p_0-1)N_0$ and variance
$\sigma^2[Z(t=0)]=4p_0(1-p_0)N_0$.

\bigskip

(ii) We select randomly the precise number $N_+(t=0)=[p_0N_0]$ of
clusters ($[x]$ is the integer part of $x$) and assign to those
the velocity $+1$, while we assign the velocity $-1$ to the
reminder $N_-(t=0)=N_0-N_+(t=0)$ clusters. In this case the
initial distribution has no dispersion, or
$\sigma[{N_-}(t=0)]=\sigma[{N_+}(t=0)]=\sigma[Z(t=0)]=0$.

\bigskip

Whatever the initial condition, the dynamical process evolves in the same
way. Clusters move with their assigned velocity and, when two clusters collide, they coagulate with probability $p$
(all results in this section take $p=0.1$) such that a larger
cluster containing all particles of both clusters is formed and
the total number of clusters is reduced by one. The new aggregated cluster
moves right or left randomly with probability $1/2$. When two clusters do not coagulate, they simply pass through each other keeping their velocities. The process
is repeated until no more collisions are possible. This could
happen because only one cluster $N=1$ (containing all particles)
remains, or because all remaining clusters move with the same
speed, either $+1$ or $-1$. The process is repeated $M$ times ($M$ realizations)
starting with different random locations and velocities of the
$N_0$ clusters. We denote by $N(t)$ the random variable that counts the number of remaining clusters at time $t$.

\subsection{Results}

\subsubsection{Number and distribution of final clusters}

Let us define the random variable $t_\infty$ as the time it takes a particular
realization to reach a steady state in which no further evolution
is possible. The remaining number of clusters at this time is denoted by $N_{\infty}=N(t>t_\infty)$. In
Table~\ref{table1} we present the average and the standard deviation of both
quantities as a function of the initial number of clusters $N_0$.
We also show in the table the probability that $N_\infty=1$, i.e. that there remains a single
cluster containing all $N_0$ particles at the end of the run.

\begin{table}[h]
\begin{center}
\begin{tabular}{|c|c|c|c|c|c|}
\hline
\textrm{$N_0$} & \textrm{$ \langle N_\infty\rangle$} & \textrm{$\sigma[N_\infty]$} & \textrm{$\langle t_\infty\rangle$} & \textrm{ $\sigma[t_{\infty}]$} & \textrm{$p(N_\infty=1)$} \\
\hline
$10$ & 3.525 & 2.014 & 4.251 & 2.988 & 1.856$\times 10^{-1}$\\
$10^2$ & 11.270 & 7.855 & 3.435 & 2.651 & 5.638$\times 10^{-2}$\\
$10^3$ & 35.684& 26.290 & 1.889 & 1.9710 & 1.793$\times 10^{-2}$\\
$10^4$ & 112.86 & 84.676 & 0.9503 & 1.305 & 5.543$\times 10^{-3}$\\
$10^5$ & 355.65 & 267.02 & 0.5436 & 0.8072 & 1.550$\times 10^{-3}$\\
\hline
\end{tabular}
\bigskip
\caption{Average and standard deviation of the final number of clusters $N_\infty$ and the time $t_\infty$ it takes to reach a state where no further evolution is possible, as a function of the initial number of clusters $N_0$. These results are for the case in which the number of
initial clusters with velocity $+1$ follows a binomial distribution with a probability $p_0=1/2$.
Averages are over $M=10^5$ realizations for
$N_0=10,10^2,10^3,10^4$ and $M=4\times 10^4$ realizations for
$N_0=10^5$. The last column is the probability that there is a
single cluster at the end of the run.\label{table1}}
\end{center}
\end{table}

\begin{figure}
\includegraphics[scale=0.38,clip]{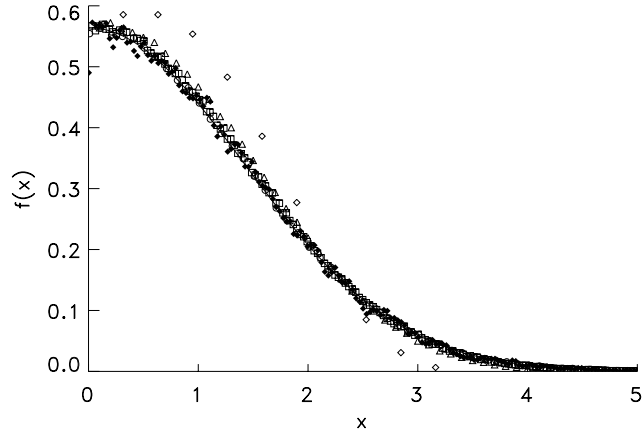}
\caption{Probability distribution of the number $N_{\infty}$ of
remaining clusters, $f(N_{\infty})$ for the same cases as in
Table~\ref{table1}: $N_0=10$, empty rhombi, $N_0=10^2$, triangles, $N_0=10^3$, squares, $N_0=10^4$, circles and $N_0=10^5$, filled rhombi. According to the scaling law discussed in the text we plot the distribution function for the variable $x=N_\infty N_0^{-1/2}$. The vertical axis is then $f(x)=N_0^{1/2}f(N_{\infty)}$. The solid line is the fit to the half-Gaussian distribution, as given by Eq.(\ref{halfgaussian}).
\label{fig1}}
\end{figure}

\begin{figure}
\includegraphics[scale=0.38,clip]{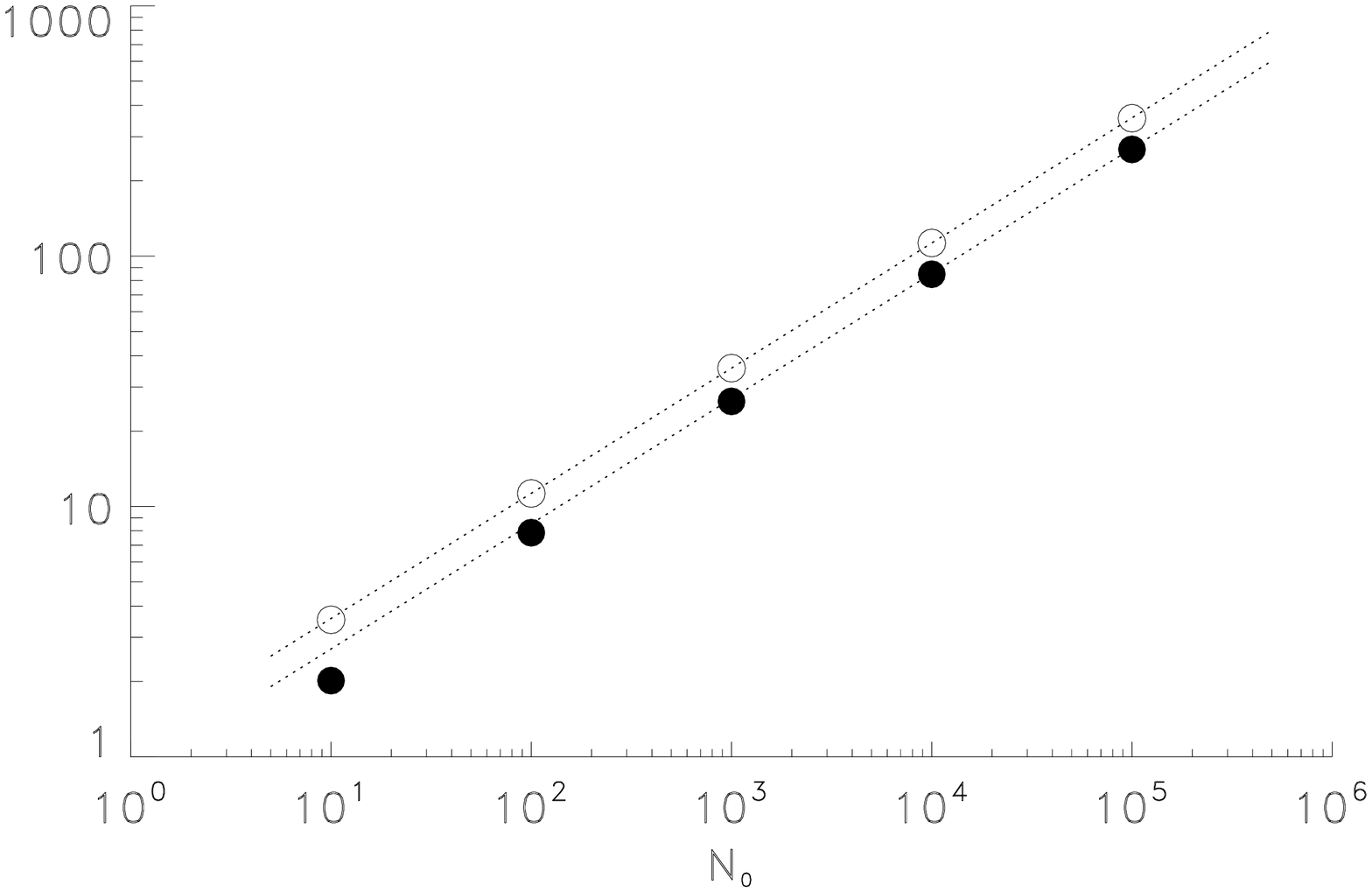}
\caption{Average value of the final number of clusters $\langle
N_\infty\rangle$ (empty symbols) and r.~m.~s. $\sigma[N_\infty]$ (filled symbols)
from Table~\ref{table1} as a function of the initial number of
clusters $N_0$, as well as the analytical expressions
in~(\ref{eq:asymp}) (dotted lines). Note that the theoretical expressions agree very well with the numerical results,
specially for large $N_0$. \label{fig2}}
\end{figure}

We have also computed from the simulations the probability distribution of the number
$N_{\infty}$ of remaining clusters, $f(N_{\infty})$. It turns out, see evidence in Fig. \ref{fig1}, that the dependence of $f(N_{\infty})$ on the intitial number of particles $N_0$ can be described by the scaling law:
\begin{equation}
\label{fN} f(N_{\infty})=N_0^{-1/2}G_+\left(N_\infty\cdot
N_0^{-1/2}\right).
\end{equation}
Moreover, $G_+(x)$ can be fitted to a half-Gaussian distribution:
\begin{equation}
G_+(x)=\frac{2}{\sigma\sqrt{2\pi}}e^{-x^2/2\sigma^2},
\hspace{2.0cm}x\in[0,\infty)
\label{halfgaussian}
\end{equation}
and $\sigma=\sqrt{2}$. This implies a mean value and standard deviation:
\begin{equation}
\langle
N_{\infty}\rangle=\frac{2}{\sqrt{\pi}}N_0^{1/2},\hspace{1.0cm}
\sigma[N_{\infty}]=\sqrt{2-\frac{4}{\pi}}N_0^{1/2}.
\label{eq:asymp}
\end{equation}
Both expressions are in good agreement with the numerical results
of the simulation, specially for large $N_0$, see Fig.
\ref{fig2}.

\begin{figure}
\includegraphics[scale=0.38,clip]{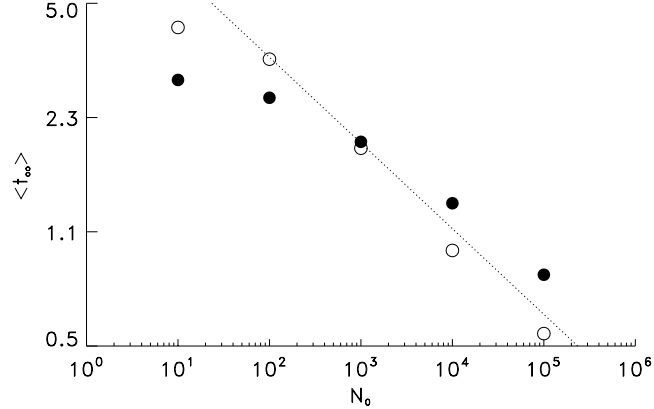}
\caption{Average value $\langle t_\infty\rangle$ (empty symbols) and
r.~m.~s. $\sigma[N_\infty
]$ (filled symbols) of the time needed to
reach the final state as a function of the initial number of
clusters $N_0$. The solid line is a fit of the form $\langle
t_\infty\rangle\sim N_0^{-a}$ with $a=1/4$.\label{fig3}}
\end{figure}

As indicated in Table~\ref{table1}, the time needed to reach equilibrium decreases with $N_0$. This {\sl a priori }Êsurprising result indicates that
the more particles there are initially, the faster the
final state is reached. Obviously, as the inital density of particles decreases with increasing $N_0$, particles are initially closer on average as $N_0$ increases and more collisions and coagulations are produced in the initial stages, so explaining this paradoxical result. The data, see Fig. \ref{fig3}, seem to suggest a
power-law relationship of the form $\langle t_\infty\rangle\sim
N_0^{-1/4}$. However, as for large $N_0$ the fluctuations are
larger than the mean value, $\sigma[t_{\infty}]>\langle
t_\infty\rangle$, the average value itself does not make much sense as a representative time to reach the asymptotic regime.

\begin{figure}
\includegraphics[scale=0.38]{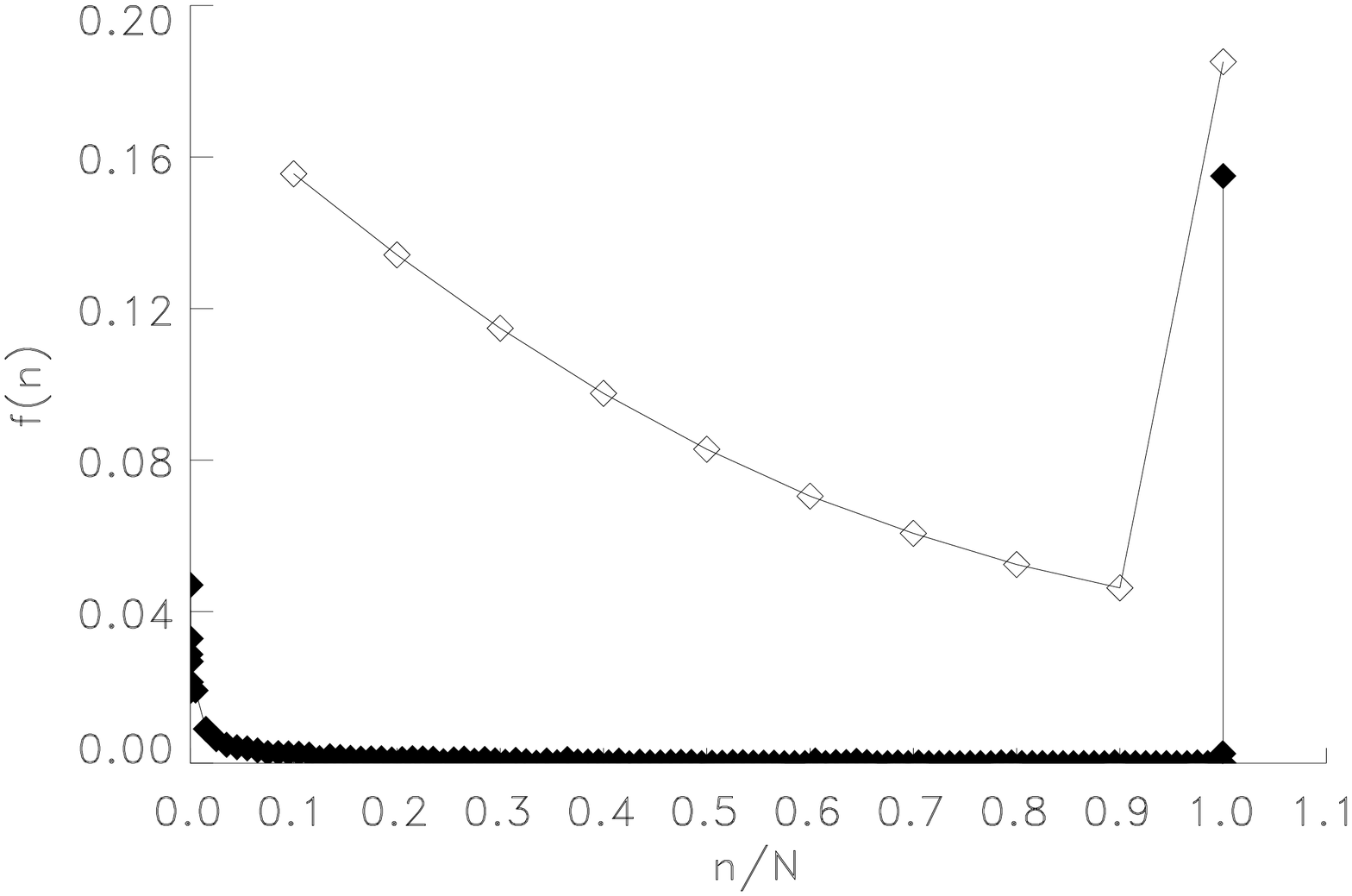}
\caption{Probability distribution $f(n)$ of the sizes of clusters
at equilibrium for $N_0=10$ (empty symbols) and $N_0=10^5$ (filled symbols, for clarity, in this case the vertical axis has been rescaled by an arbitrary factor of $100$). Note
the maximum at $n=N_0$.\label{fig4}}
\end{figure}

Finally, we have computed the probability distribution $f(n)$ of
the sizes of clusters at equilibrium. This is defined as the probability that a particle belongs
to a cluster of size $n$. More precisely, $f(n)$ is computed as the average (over realizations)
number of clusters of size $n$ multiplied by $n$ and divided by
$N_0$, such that the normalization condition is
\begin{equation}
\sum_{n=1}^{N_0}f(n)=1.
\end{equation}
The results are plotted in Figs.~\ref{fig4} for $N_0=10$ and
$N_0=10^5$. Intermediate values of $N_0$ showing a similar behavior. It is not clear from the data whether a scaling law valid for
all values of $N_0$ exists for these functions. Note
the existence of the maximum at $n=N_0$ indicating that the most probable outcome is that a particle belongs to the single cluster containing all particles. This probability can be
computed from Eq.~(\ref{fN}) taking $N_\infty=1$. In the limit of
large $N_0$ the argument of the half-Gaussian distribution is
$x=1\cdot N_0^{-1/2}\to 0$ implying that $G_+(0)=1/\sqrt{\pi}$, so that the value
at the maximum is $N_0^{-1/2}/\sqrt{\pi}$. As it can be seen in
Table~\ref{table1} and in Fig.~\ref{fig5}, this result agrees well
with the simulation results.

\begin{figure}
\includegraphics[scale=0.38]{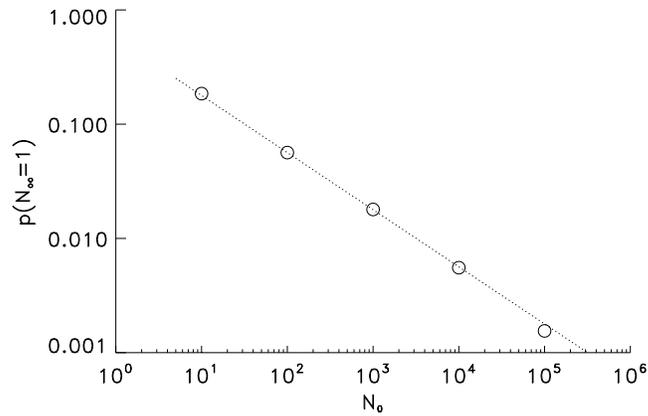}
\caption{Probability to end in a single cluster containing all
particles. With symbols we plot the simulation results from
Table~\ref{table1} and the line is the theoretical expression
$N_0^{-1/2}/\sqrt{\pi}$.\label{fig5}}
\end{figure}

\subsubsection{Time dependence of fluctuations}

\begin{figure}
\includegraphics[scale=0.38]{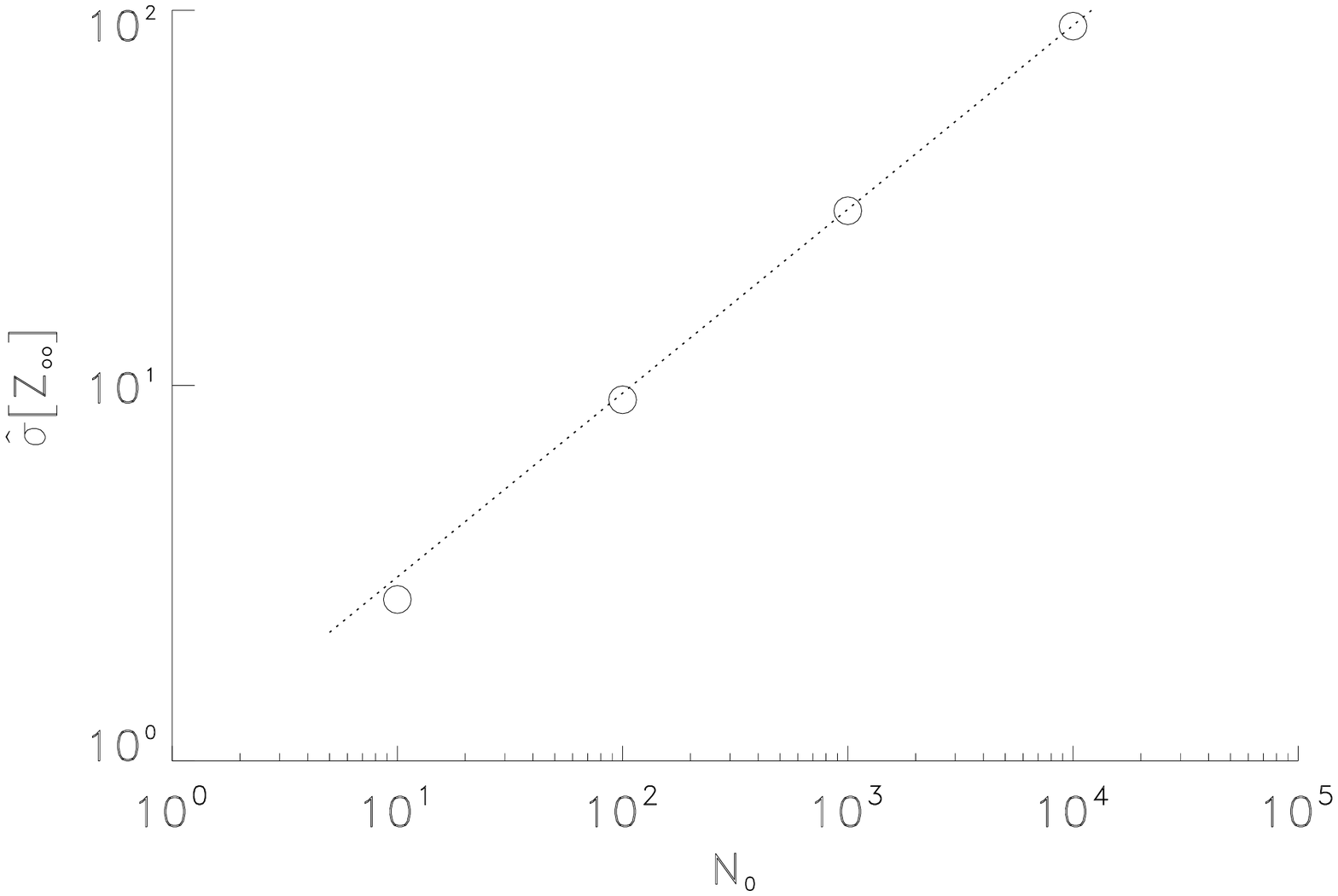}
\caption{Root mean square $\hat \sigma[Z_\infty]$ of the variable
$Z(t)=N_+(t)-N_-(t)$ at the steady state as a function of the
initial number of clusters $N_0$ for a fixed number of $p_0N_0$
clusters, $p_0=1/2$. With symbols we plot the simulation results
and the line is the function $N_0^{1/2}$.\label{fig6}}
\end{figure}

We have computed the mean value $\langle Z(t)\rangle$ and the
fluctuations $\sigma^2[Z(t)]$ of the variable $Z(t)=N_+(t)-N_-(t)$,
the difference between the number of clusters that have velocity
$+1$, $N_+(t)$ and those that have velocity $-1$, $N_-(t)$. The
data (not shown) indicates that $Z(t)$ can be well approximated by
a Gaussian distribution of mean $\langle Z(t)\rangle$ and
r.~m.~s. $\sigma[Z(t)]$. We have found that, in accordance with the
theoretical results, $\langle Z(t)\rangle$ is constant with time
whereas the fluctuations can be decomposed as
$\sigma^2[Z(t)]=\sigma^2[Z(t=0)]+\hat \sigma^2[Z(t)]$, being $\hat
\sigma^2[Z(t)]$ independent on whether the initial condition for
the number of clusters with velocity $+1$ was a fixed number
$[p_0N_0]$ ($\sigma[Z(t=0)]=0$) or it followed a binomial
distribution ($\sigma^2[Z(t=0)]=4p_0(1-p_0)N_0$). In the steady
state, the variable $Z(t)$ takes the constant value $Z_\infty$ and
the fluctuations of the variable $Z_\infty$, within the numerical
precision, are well approximated by $\hat\sigma^2[Z_\infty]=2\bar
p_0 N_0$ with $\bar p_0=\min(p_0,1-p_0)$, see evidence in Fig.~\ref{fig6} for $p_0=1/2$.

\begin{figure}
\includegraphics[scale=0.22]{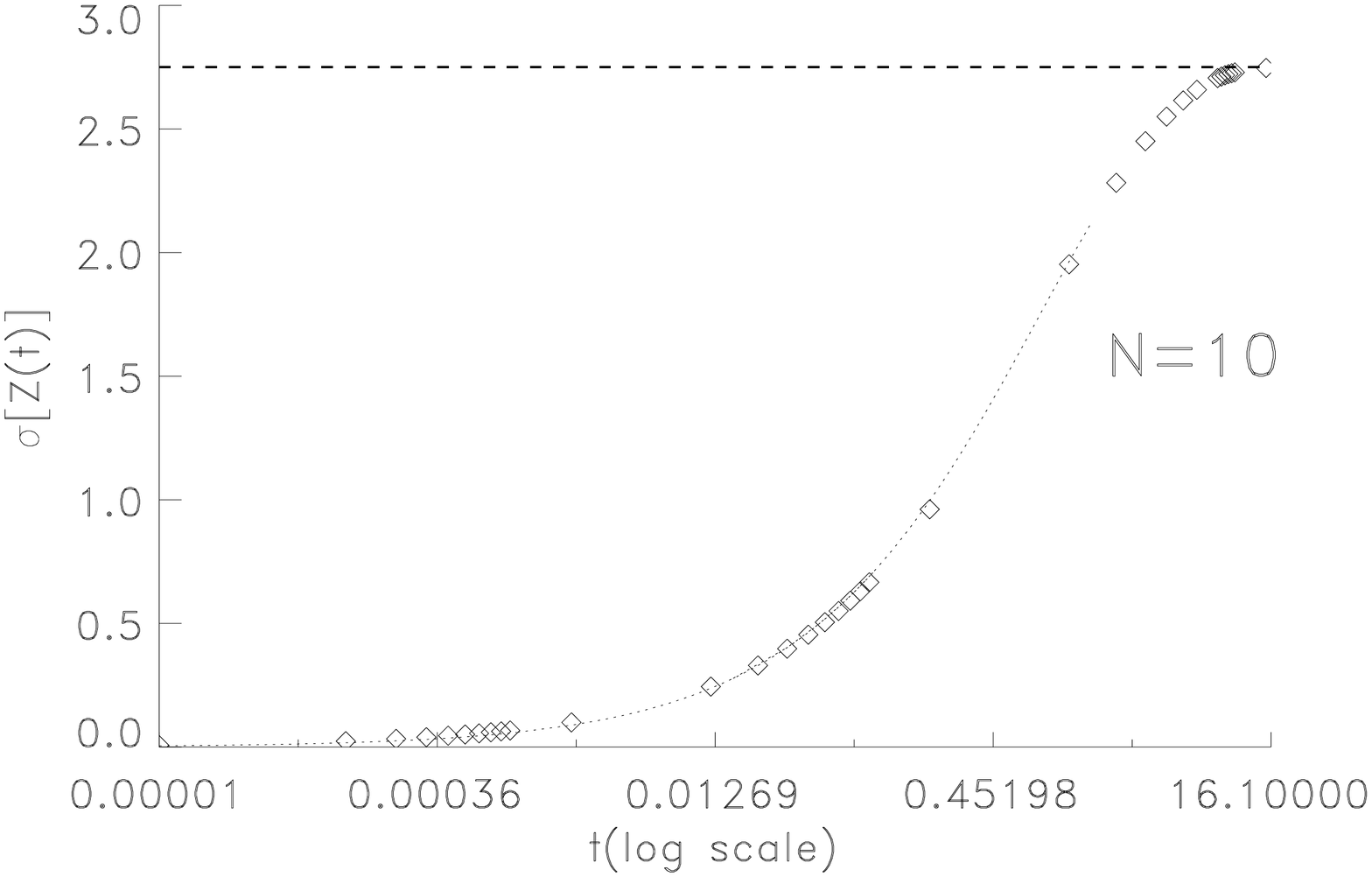}\includegraphics[scale=0.22]{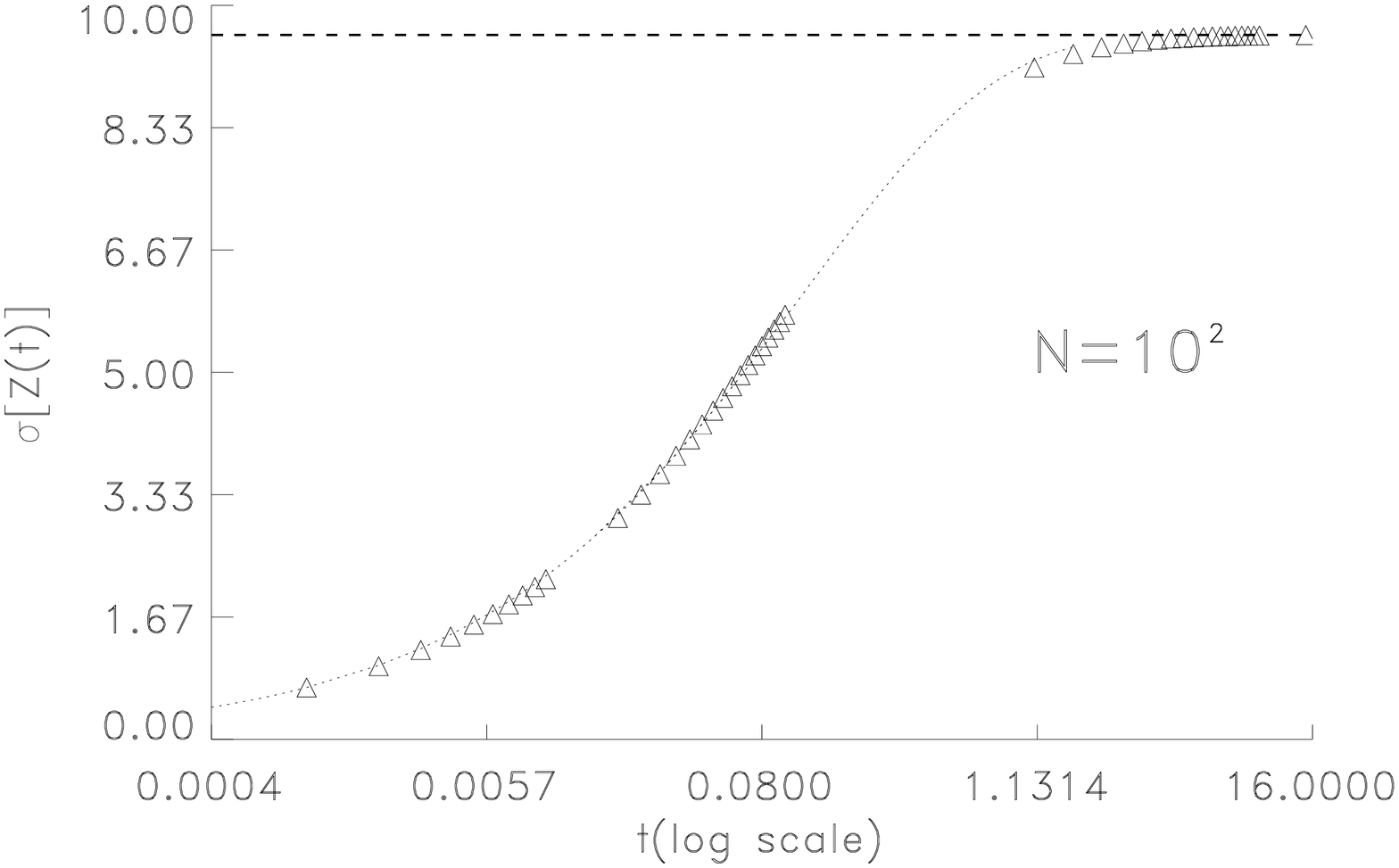}
\includegraphics[scale=0.22]{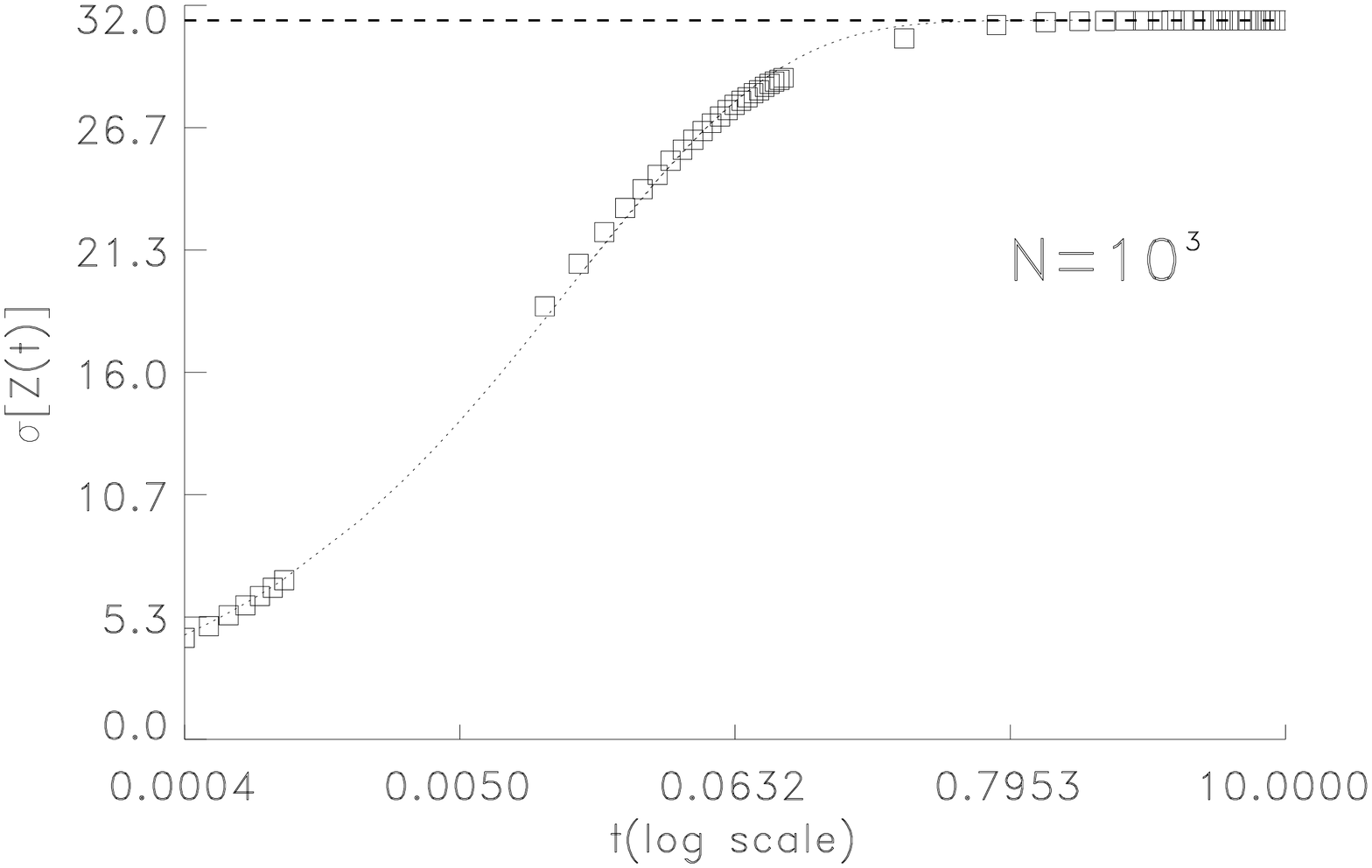}\includegraphics[scale=0.22]{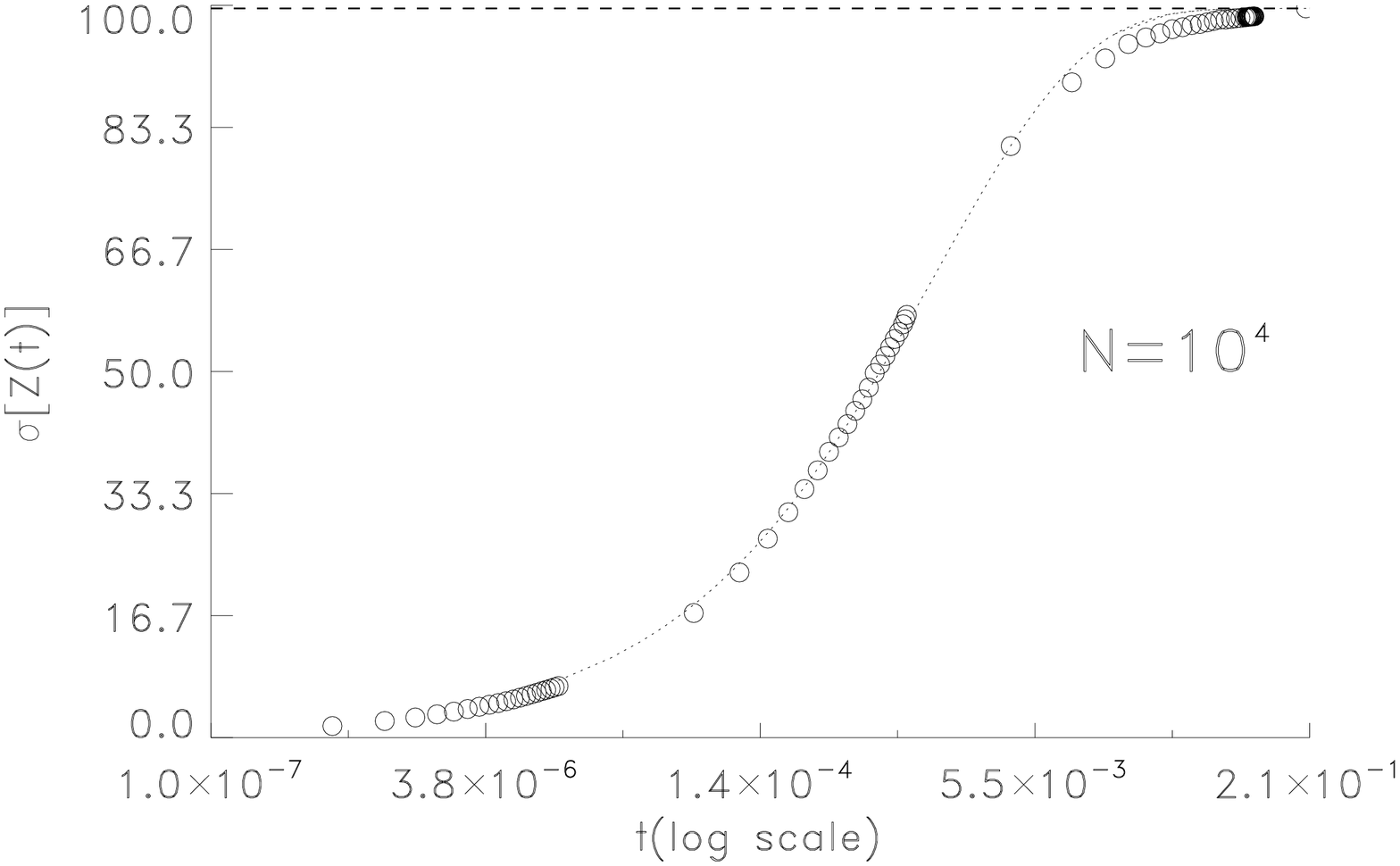}
\includegraphics[scale=0.35]{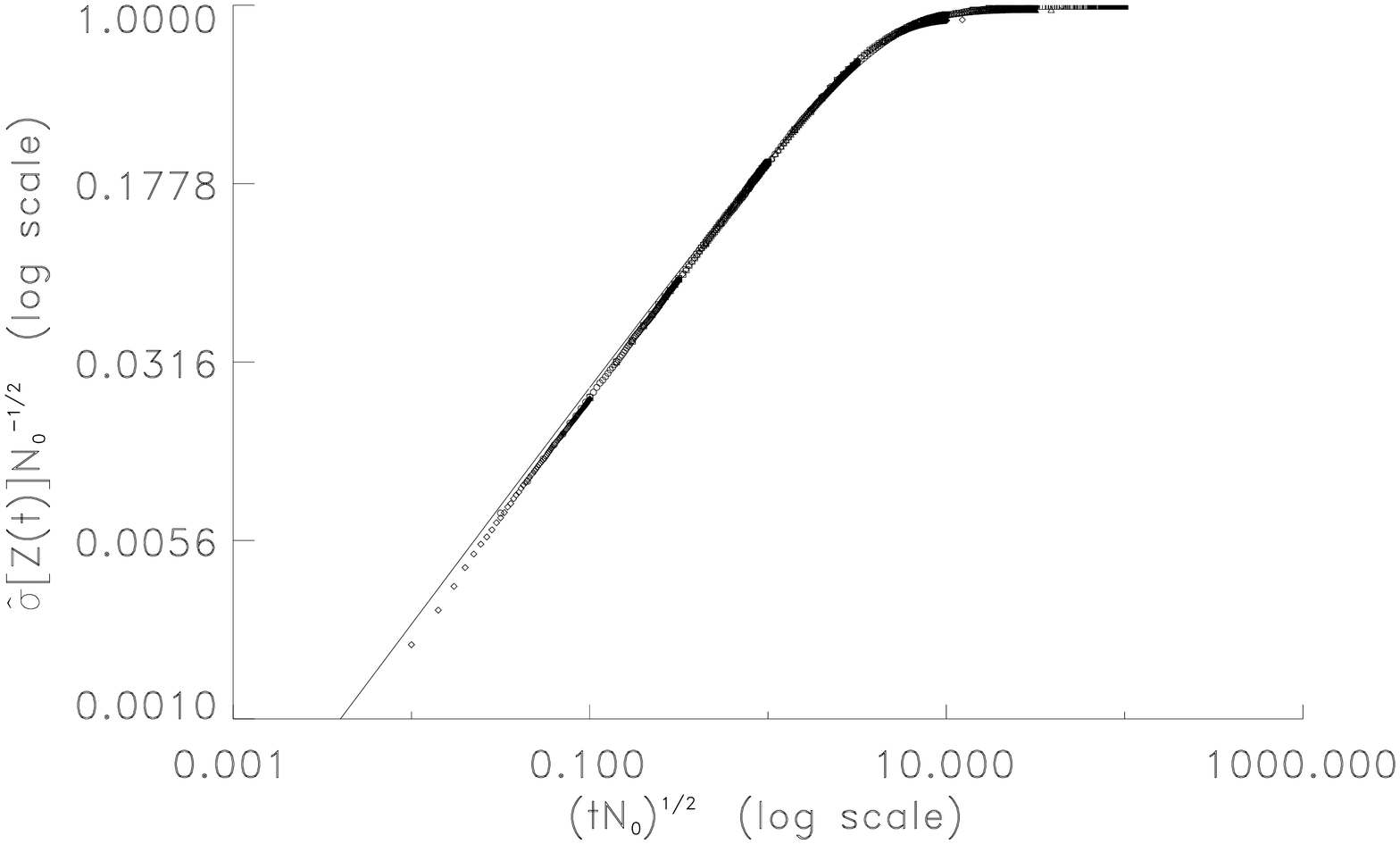}
\caption{Time evolution of the r.~m.~s. $\hat \sigma[Z(t)]$ of the
variable $Z(t)=N_+(t)-N_-(t)$ as a function of time for different values of $N_0$. In the top four panels, we plot with symbols the simulation
results and the line is the empirical function given by
Eq.~(\ref{eq:empirical}) with $(\beta,c)=(1.12,0.58),
(3.16,0.54), (9.17,0.52), (27.0,0.50)$ for
$N_0=10,10^2,10^3,10^4$, respectively. This suggests $c=1/2$ and
$\beta=a\sqrt{ N_0}$, or $ \hat
\sigma[Z(t)]=\sqrt{4p_0(1-p_0)N_0}(1-e^{-a\sqrt{N_0t}})$. The horizontal dotted
lines are the (observed) values for $\hat\sigma[Z_\infty]$. In the bottom panel we replot the data for all values of $N_0$ (same symbols) using the rescaling indicated by Eq.(\ref{fluctuations}). The solid line is the curve $1-e^{-0.25 x}$,
\label{fig7}}
\end{figure}

Finally, we have found numerical evidence indicating that the time dependence of the
fluctuations can be fitted to the form
\begin{equation}
\hat \sigma[Z(t)]=\hat\sigma[Z_\infty](1-e^{-\beta t^c}).
\label{eq:empirical}
\end{equation}
The constants $\beta$ and $c$ depend on $N_0$ and have been fitted using the results of
the simulations, see caption of Fig.~\ref{fig7} for their numerical values  for $p_0=1/2$.
These results suggest the following functional form
\begin{equation}\label{fluctuations}
\hat \sigma[Z(t)]=\sqrt{4p_0(1-p_0)N_0}\left(1-e^{-a\sqrt{N_0t}}\right).
\end{equation}
This expression has been checked in the case $p_0=1/2$ by plotting $\hat \sigma[Z(t)] N_0^{-1/2}$ vs $\sqrt{N_0t}$, see bottom panel of  Fig.\ref{fig7}, with a value $a\sim0.25$.

\subsection{Connection with theoretical results}

It is obvious that the results of the simulations go much further
beyond those of the kinetic theory. Yet, we have found agreement in
two theoretical predictions: 1) the conservation in time of
$\left\langle Z(t) \right\rangle$ and 2) the trend towards the
formation of a cluster containing all particles in the system for
$p_0=1/2$.

It is difficult to compare many more predictions apart from these
two, because the kinetic approach neglects many sources of
fluctuations by its very nature, while the simulations retain all of
them. Anyway, we have found result~(\ref{fluctuations}) especially
interesting and, although a direct comparison with the theory is not
possible, we will offer a possible theoretical explanation of it
using quantities that can theoretically accessed.

Eq.~(\ref{fluctuations}) can be expanded around $t=0$ to find
\begin{equation}
\hat \sigma[Z(t)] = a  N_0 \sqrt{t}.
\end{equation}
This suggests the random variable $Z(t)$, in the limit $t \to 0^+$,
performs a random walk with zero mean and diffusion constant $a^2
N_0^2/\approx N_0/16$. Then the diffusive behavior is modified due
to saturation effects as the system approaches the ordered state
giving rise to the behavior described in~(\ref{fluctuations}).

Now we will calculate a couple of quantities that might result of
interest to interpret the value of the above found diffusion
constant. From now on the condition $N^+(0)=N^-(0)$ on the initial
condition is imposed and all the clusters are initially of size
$\ell=1$ (same conditions as in the simulations). First of all we
consider the mean cluster size of all clusters traveling in the $+$
direction (identical formulas hold for the other direction), which
is given by the formula
\begin{equation}
\frac{M^+(t)}{N^+(t)}= 1 + \frac{N_0}{4}t.
\end{equation}
The second quantity of interest is the average size of the cluster
the particles belong to, again for the $+$ direction. It is given by
the formula
\begin{equation}
\frac{\sum_{\ell=1}^\infty \ell^2 f^+(\ell,t)}{M^+(t)}= 1 +
\frac{N_0}{2}t.
\end{equation}
In both cases the calculations are performed following the
techniques introduced in section~\ref{introduction}. The second
result means the particles are progressively located in clusters of
bigger size, and per unit of time this size increases in $N_0/2$ in
average. This factor describes the average number of particles
traveling in the opposite direction, and correspondingly it is the
number of possible collisions for a given particle. The first result
indicates how the mean cluster size grows in time. Per unit time
this size increases by a factor $(1/2)\times(N_0/2)$. Again we find that the mean
cluster size increases with a velocity proportional to the
number of possible collisions, but this time the proportionality
factor $1/2$ signals that after every collision one cluster
disappears and the newborn cluster chooses its direction of motion
randomly.

In view of these results it is tempting to interpret the value of
the diffusion constant $(N_0/4)\times(N_0/4)$ as the square of the
rate at which the mean cluster size increases.

\section*{Acknowledgments}

CE is grateful to the IFISC for its hospitality.
RT acknowledges financial support from MINECO (Spain), Comunitat Aut\`onoma de les Illes Balears, FEDER, and the European Commission
under project FIS2012-30634.

\medskip
Received xxxx 20xx; revised xxxx 20xx.
\medskip

\end{document}